\documentclass[12pt, a4paper]{amsart}
\usepackage[colorlinks=true, breaklinks=true, linkcolor=black, citecolor=blue, urlcolor=red]{hyperref} 
\usepackage{amscd,amsmath,amssymb,amsthm,amsfonts, amstext, mathrsfs,longtable,tikz}
\newcommand{\DiagS}{\begin{tikzpicture}[baseline = .3cm]
\draw (0,-1/3)--(0,0);
\draw (0,0)--(1,1);
\draw (0,0)--(-1,1);
\draw (1/3,1/3)--(-1/3,1);
\draw (2/3,2/3)--(1/3,1);
\end{tikzpicture}
}
\newcommand{\DiagT}{\begin{tikzpicture}[baseline = .3cm]
\draw (0,-1/3)--(0,0);
\draw (0,0)--(1,1);
\draw (0,0)--(-1,1);
\draw (-1/3,1/3)--(1/3,1);
\draw (-2/3,2/3)--(-1/3,1);
\end{tikzpicture}
}

\newcommand{\Diagetwo}{\begin{tikzpicture}[baseline = .3cm]
\draw (0,0)--(0,1.25);
\draw (0,.25)--(.5,.75);
\draw (0,.25)--(-1,1.25);
\draw (-1,1.25)--(-1,1.5);
\draw (-.5,.75)--(0,1.25);
\draw (.5,.75)--(0,1.25);
\draw (-.5,.75)--(-.5,1.25);
\draw (0,1.25) arc (0 : 180 : .25) ;
\node at (-.25,.25) {\tiny{$R$}};
\node at (-.75,.75) {\tiny{$R$}};
\node at (.25,1.35) {\tiny{$R^*$}};
\end{tikzpicture}
}

\newcommand{\Diagethree}{\begin{tikzpicture}[baseline = .3cm]
\draw (0,0)--(0,1.25);
\draw (0,.25)--(.5,.75);
\draw (0,.25)--(-1.5,1.75);
\draw (-1.5,1.75)--(-1.5,2);
\draw (-1,1.25)--(-.5,1.75);
\draw (.5,.75)--(-.5,1.75);
\draw (-.5,.75)--(0,1.25);
\draw (-.5,.75)--(-.5,1.75);
\draw (-1,1.25)--(-1,1.75);
\draw (-.5,1.75) arc (0 : 180 : .25) ;
\node at (-.25,.25) {\tiny{$R$}};
\node at (-.75,.75) {\tiny{$R$}};
\node at (-1.25,1.25) {\tiny{$R$}};
\node at (.25,1.35) {\tiny{$R^*$}};
\node at (-.25,1.85) {\tiny{$R^*$}};
\end{tikzpicture}
}

\def\Aut{\operatorname{Aut}}
\def\Ad{\operatorname{Ad}}
\newcommand{\act}{\curvearrowright}
\newcommand{\bu}{\bullet}

\newcommand{\CC}[0]{\mathcal{C}}
\DeclareMathOperator{\Comm}{Comm}
\newcommand{\C}{\mathbf C}
\newcommand{\cF}{\mathcal F}
\newcommand{\Ga}{\Gamma}
\newcommand{\ga}{\gamma}
\newcommand{\fH}{\mathfrak H}
\newcommand{\scrH}{\mathscr H}
\DeclareMathOperator{\Hilb}{Hilb}
\def\id{\operatorname{id}}
\newcommand{\scrK}{\mathscr K}
\newcommand{\la}{\lambda}
\newcommand{\La}{\Lambda}
\def\Loop{\operatorname{loop}}
\newcommand{\N}{\mathbf N}
\newcommand{\cP}{\mathcal P}\newcommand{\CP}[0]{\mathcal{P}}
\newcommand{\R}{\mathbf R}
\DeclareMathOperator{\Rec}{Rec}

\newcommand{\fT}{\mathfrak T}
\DeclareMathOperator{\target}{target}
\DeclareMathOperator{\TLJ}{TLJ}
\newcommand{\cU}{\mathcal U}
\newcommand{\varep}{\varepsilon}
\newcommand{\Z}{\mathbf Z}

\newcommand{\PA}{planar algebra}
\newcommand{\SPA}{subfactor planar algebra}

\begin{document}

\newtheorem{theorem}{Theorem}[section]
\newtheorem{proposition}[theorem]{Proposition}
\newtheorem{lemma}[theorem]{Lemma}
\newtheorem{definition}[theorem]{Definition}
\newtheorem{problem}[theorem]{Problem}
\newtheorem{corollary}[theorem]{Corollary}
\newtheorem{remark}[theorem]{Remark}
\newtheorem{question}[theorem]{Question}
\newtheorem{conjecture}[theorem]{Conjecture}
\newtheorem{notation}[theorem]{Notation}

\title[
Jones representations of Thompson's group $F$ arising from TLJ algebras]{
Jones representations of Thompson's group $F$ arising from Temperley-Lieb-Jones algebras}

\author{Valeriano Aiello, Arnaud Brothier, \& Roberto Conti}  

\address{
Section de Math\'ematiques
Universit\'e de Gen\`eve
2-4 rue du Li\`evre, Case Postale 64, 1211 Gen\`eve 4, Suisse} 
\email{valerianoaiello@gmail.com}

\address{School of Mathematics and Statistics, University of New South Wales, Sydney NSW 2052, Australia, The Red Centre, East Wing, Room 6107} 
\email{arnaud.brothier@gmail.com\endgraf
\url{https://sites.google.com/site/arnaudbrothier/}}

\address{Dipartimento di Scienze di Base e Applicate per l'Ingegneria, 
Sapienza Universit\`a di Roma, Via A. Scarpa 16, 00161 Roma, Italy}
\email{roberto.conti@sbai.uniroma1.it}

\thanks{Valeriano Aiello acknowledges   support of   
 the Swiss National Science Foundation. Roberto Conti acknowledges partial support by Sapienza Universit\`a di Roma. 
 Arnaud Brothier was supported by European Research Council Advanced Grant 669240 QUEST and is now supported by a University of New South Wales Sydney starting grant.}
 \maketitle
\begin{abstract}
Following a procedure due to V. Jones, using suitably
normalized elements in a Temperley-Lieb-Jones (planar) algebra
we introduce a 3-parametric family of unitary representations of
the Thompson's group $F$ equipped with canonical (vacuum) vectors and study
some of their properties. In particular, we discuss the behaviour
at infinity of their matrix coefficients, thus showing that these
representations do not contain any finite type component.
We then focus on a particular representation known to be quasi-regular and irreducible and show that it is inequivalent to itself once composed with a classical automorphism of F. This allows us to distinguish three 
equivalence classes in our family.
Finally,  we investigate a family of stabilizer subgroups of $F$ indexed by subfactor Jones indices that are described in terms of the chromatic polynomial. In contrast to the first non-trivial index value for which the corresponding subgroup is 
isomorphic to the Brown-Thompson's group $F_3$, we show that when the index is large enough this subgroup is always trivial.

 \end{abstract}


\section*{Introduction and main results}
R. Thompson's group $F$ is one of the most fascinating
countable discrete groups, yet very mysterious for the study of its 
analytical properties has been challenging experts for decades. We refer to \cite{CFP96} for a nice introduction to the basic facts about $F$ and its relatives $T$ and $V$.

It is known that $F$ is inner amenable \cite{Jo97} and it has the Haagerup approximation property \cite{Fa03}, both of which express some weakened form of amenability.
However, somewhat surprisingly, despite the many attempts the question about the amenability of $F$ still remains unanswered, along with exactness (and thus weak amenability) and soficity. 

Several approximation properties for groups are based on suitable asymptotic behaviour of matrix coefficients of unitary representations. For this reason, it is of some interest to determine as much as possible about the representation theory of $F$. Earlier studies of unitary representations of $F$ appear in \cite{Gar12},  \cite{DuMe14} and \cite{Ole16}.
In this work, we follow the general procedure introduced by V.F.R. Jones in  \cite{Jo14,Jo16}  where a large family of unitary representations of $F$ (and $T$) are built 
using Jones' planar algebras or, more generally, a category/functor method.  
Some of these representations admit easy algorithms for computing matrix coefficients leading to new proofs regarding analytical properties of the Thompson's groups \cite{BJ18} but also computation of limits of rotations \cite{BJ18b}. They also provide an explicit connection with the Cuntz algebra in the spirit of the work of Nekrashevych \cite{Ne04}.

In this work we consider only $F$  
and the very first construction of Jones involving a planar algebra. 
We take some step towards an understanding of the main features of a class of such representations depending on three different parameters, one selecting a Temperley-Lieb-Jones (planar) algebra and two more determining a normalized element in that algebra expressed as a linear combination of the identity and a non-trivial TLJ-generator.
More precisely, given a loop parameter $\delta \in  \{2\cos(\pi/n) , \ n=4,5,6,\ldots\} \cup [2,\infty)$
and a normalized 2-box $R$ characterised by two complex parameters $a,b$ satisfying a suitable equation, we obtain a unitary representation $\pi=\pi_{a,b}^\delta$ of $F$.
This representation comes along with a certain vacuum vector $\Omega$ and thus a vacuum state.\\
We prove that for any choice of $\delta$ and non-zero real parameters $a,b$ as above the corresponding representations do not admit any finite-dimensional subrepresentation
and moreover their matrix coefficients do not vanish at infinity.
This result, combined with a theorem of Dudko-Medynets \cite{DuMe14}, implies that these representations do not contain any finite type components (of type I$_n$ and II$_1$).\\
For any loop parameter $\delta$, the Jones-Wenzl idempotent (properly rescaled) gives us a normalized element $R$ of TLJ.
Jones called it the {\it chromatic choice} because the vacuum state 
applied to
$g\in F$ can be expressed in terms of the chromatic polynomial of a certain graph $\Gamma_g$ associated to $g$ 
evaluated at $\delta^2.$
Jones showed that at $\delta=\sqrt 2$  the 
stabilizer subgroup of the vacuum vector is nothing but the set of those elements $g \in F$ for which the graph $\Gamma_g$ is bipartite.
It is a remarkable subgroup $\Vec F<F$ called the Jones subgroup, that turns out to be isomorphic to $F_3$ \cite[Lemma 4.7.]{GS17}.
Moreover, the representation $\pi$ for this choice is unitarily 
equivalent to the quasi-regular representation $\lambda_{F/ \Vec F}$ associated to $\Vec F<F.$ 
In passing, we mention that many other graph and link invariants may be interpreted as vector states associated with unitary representations of both the Thompson's group and the Jones oriented subgroup. We refer, for instance, to \cite{ACJ}, where this is shown for $\vec F$ and   the HOMFLYPT polynomial, and to \cite{AiCo1, AiCo2} where this is done with different and more elementary (but less powerful) methods.
\\
By flipping the two parameters $a$, $b$, we obtain a new representation that is unitarily equivalent to $\lambda_{F/\sigma_F(\Vec F)}$,
where $\sigma_F$ is the automorphism of $F$ associated to the homeomorphism $\sigma(t)=1-t$ of the unit interval $[0,1]$.
We prove that $\Vec F$ and $\sigma_F(\Vec F)$ are not quasi-conjugate.
By Golan-Sapir \cite{GS17}, $\Vec F$ is equal to its own commensurator, implying that the quasi-regular representation $\lambda_{F/\Vec F}$ is irreducible.
This, together with a classical argument going back to Mackey, shows that the quasi-regular representations associated to $\Vec F < F$ and $\sigma_F(\Vec F) < F$ are not unitary equivalent, thus providing two unitarily inequivalent irreducible representations $\pi$ and $\pi \circ \sigma_F$ in our family.
Choosing another representation $\pi'$ with parameters $a$ and $b$ equal to each other, we get that $\pi'$ and $\pi'\circ\sigma_F$ are unitary equivalent, implying that $\pi'$ is unitary equivalent neither to $\pi$ nor to $\pi\circ\sigma_F$, thus ensuring that our family of representations contains at least three distinct   classes.
\\
Then we 
give a closer look at the stabilizer  subgroups $F^\delta<F$ of the vacuum vector for the chromatic choice but for any value of $\delta.$
We prove that for $\delta$ large enough ($\delta>2.16$) or $\delta \in \{\sqrt 3 , 2\}$, the subgroup $F^\delta$ is actually trivial.
Last but not the least, for any choice of $\pi_{a,b}^\delta$ with real and non-zero parameters $a,b$ the vacuum state can be expressed in terms of the Tutte polynomial suggesting that similar arguments could prove that the stabilizer of the vacuum is generically trivial. We have not investigated this direction yet. 

There are some important questions that remain open.
\begin{enumerate}
\item Do we have that in most of the cases considered in Section \ref{sec:stab} the stabilizer subgroup of the vacuum vector is trivial? So far the only case for which the stabilizer is known to be non-trivial appears when $\delta=\sqrt 2.$
\item Is $\pi$ always irreducible? This is only known to be true again when $\delta=\sqrt 2$ and we are in the chromatic choice or when we compose this representation with $\sigma_F.$
\item Does the considered family of unitary representations contains infinitely (perhaps uncountably) many distinct classes? So far, we could only distinguish three mutually inequivalent non-trivial representations.
\end{enumerate}

The content of the paper is as follows.
In Section \ref{sec-Preliminaries}  we collect some preliminaries to be used throughout the whole text.
In Section \ref{sec-First-results} we discuss some useful facts about unitary equivalence 
of the Jones representations.
In Section \ref{sec-TLS-reps}  we introduce the main class of representations studied in this paper and prove a number of results about their matrix coefficients.
In Section \ref{sec-two-reps} we exhibit a pair of  quasi-regular Jones representations that are inequivalent.
In Section \ref{sec:stab} we focus on a natural class of representations related to the chromatic polynomial, and show that for these representations the stabilizer subgroup of the vacuum vector is trivial in many cases.

\section{Preliminaries}\label{sec-Preliminaries}
\subsection{Subfactors and planar algebras}
A subfactor $N\subset M$ is a unital inclusion of type II$_1$ factors.
Its Jones index $[M:N]$ is the Murray-von Neumann dimension of $L^2(M)$ as a left $N$-module where $L^2(M)$ is the Gelfand-Naimark-Segal Hilbert space associated to the unique faithful normal tracial state of $M$ \cite{Jo83}.
The celebrated index rigidity theorem of Jones claims that the range of Jones indices is exactly equal to the following set
$$\{ 4\cos(\pi/n)^2 : n\geq 3\} \cup [4,\infty].$$
The standard invariant of $N\subset M$ where $[M:N]$ is finite is the lattice of relative commutants $M_i'\cap M_j$ for $i\leq j$ and where $M_{-1}=N\subset M_0=M\subset M_1\subset \cdots$ is the Jones tower obtained by iterating the Jones basic construction.
It has been axiomatized by Popa as a $\lambda$-lattice and later on by Jones as a subfactor planar algebra \cite{P95, Jo99}.
Note that Ocneanu gave another axiomatization for finite depth subfactors, i.e.~when the $N$-bimodule tensor category generated by $L^2(M)$ has finitely many equivalence classes of irreducible objects \cite{O89}.

We are interested in subfactor planar algebras that we briefly define.
We refer the reader to \cite{Jo99} for details.
A shaded C*-planar algebra $\cP=(\cP_n^\pm)_{n\geq 0}$ is a collection of finite dimensional C*-algebras $\cP_n^\pm$ on which the operad of shaded planar tangles acts.
We assume that $\cP_0^+$ and $\cP_0^-$ are one-dimensional.
We think of an element of $\cP_n^\pm$ as a box with $2n$ boundary points, $n$ on the top and $n$ on the bottom.
The distinguished interval (the dollar sign) being at the top left corner of the box, that is in a region shaded by $\pm$.
The multiplication is then given by vertical concatenation and thus the unit of $\cP_n^\pm$ is a diagram with n vertical straight lines.
We have a unital inclusion of $\cP_n^\pm$ into $\cP_{n+1}^\pm$ by adding one vertical straight line to the right of a box.
The \PA\ admits two loop parameters $\delta_+,\delta_-$ that are the values of a closed circle in $\cP_0^+$ and $\cP_0^-$ respectively, where $\cP_0^\pm$ is identified with $\C$.
Each $\cP_n^\pm$ admits two tracial states $\tau_l,\tau_r$ that are the left and right traces.
The value $\tau_l(x)$ (resp. $\tau_r(x)$) is obtained by connecting each string of the bottom to a string of the top on the left (resp. on the right) of the box and dividing by $\delta_\pm^n$ if $x\in\cP_n^\pm$.
If each of them is faithful we say that $\cP$ is non-degenerate and we say that $\cP$ is spherical if $\tau_l=\tau_r$. We then write $\tau:=\tau_l$ and call it the trace of $\cP$.
In that case $\delta_+=\delta_-=: \delta$, that we call the loop parameter.
A \SPA\ is a non-degenerate spherical C*-planar algebra.
Note that the loop parameter $\delta$ of a subfactor planar algebra is the square root of a non-trivial finite Jones index and thus belongs to the set $\{ 2\cos(\pi/n) : n\geq 4\} \cup [2,\infty)$, where we have excluded $1$ and $\infty$.
A \SPA\ is irreducible if $\cP_1^+$ and $\cP_1^-$ are both one-dimensional.
Recall that for any $\delta$ in the above set there is a unique minimal \SPA\ with loop parameter $\delta$ that is called the Temperley-Lieb-Jones \PA\ and that we denote by $\TLJ(\delta)$ or simply $\TLJ$ if the context is clear \cite{TL71,Jo83}.
A spanning set of $\TLJ_n^\pm$ is given by all planar diagrams of non-crossing curves in a box with $n$ boundary points on the bottom and $n$ on the top.
The antilinear involution $*$ applied to such diagram is then the vertical symmetry.
Note that $\TLJ$ is an irreducible \SPA . 

We briefly define the rectangular category of a \SPA.
We refer the reader to \cite{Jo14} for details and precise definitions.
Let $\cP$ be a \SPA\ and consider its rectangular category $\Rec(\cP)$ whose collection of objects is $\N=\{1,2,\cdots\}$.
The space of morphism from $n$ to $m$ is empty if $n+m$ is odd and otherwise is equal to a copy of $\cP_{(n+m)/{2}}^+$. 
We think of an element of $\Rec(\cP)(n,m)$ as an element of $\cP_{(n+m)/{2}}^+$  that we represent as a box with $n$ points on the bottom and $m$ on the top and where the distinguished interval is placed on the top left corner. 
The composition of morphism is then obtained by concatenating vertically such diagrams.
Note that $\Rec(\cP)(n,n)$ can be canonically identified with $\cP_n^+$ as an algebra.
We equip the category $\Rec(\cP)$ with a contravariant endofunctor $\dag:\Rec(\cP)\to\Rec(\cP)$ such that $\dag(n)=n$ and $x^\dag$, with $x\in\Rec(\cP)(n,m)$ identified with an element of $\cP^+_{(n+m)/2}$, is obtained by considering the element $x^*\in \cP^+_{(n+m)/2}$ and identifying it with a box-diagram with $m$ boundary points on the bottom and $n$ on the top and where the distinguished interval is still on the top left corner.
Therefore, $x^\dag\in \Rec(\cP)(m,n)$ and we have that $(x^\dag)^\dag=x$.
This provides a sesquilinear form on $\Rec(\cP)(n,m)$ given by $\langle x, y \rangle:= \tau(y^\dag \circ x)$ where $\tau$ is the trace on $\cP_n^+.$
Since $\CP$ is non-degenerate we have that this form is an inner product and then it gives a structure of Hilbert space for $\Rec(\cP)(n,m)$ since it is complete by finite dimensionality.

\subsection{Jones representations of Thompson's group $F$}\label{sec:JonesRep}

\subsubsection{Thompson's group $F$.}
We recall some basic facts about  Thompson's group $F$. We refer the reader to \cite{CFP96} for details.
The group $F$ is defined as the set of orientation preserving homeomorphisms of the closed unit interval $[0,1]$ that are piecewise linear with finitely many breakpoints at dyadic rationals and slopes in $2^{\Z}$. In this paper, $F$ is endowed with the discrete topology. It is known that $F$ is an ICC (infinite conjugacy classes) countable group. 
The quotient of $F$ by its commutator subgroup $[F,F]$ is isomorphic to ${\Z}^2$. Moreover, any proper quotient of $F$ is abelian.
It also known that an irreducible finite dimensional representation of $F$ must be necessarily of dimension one \cite{DuMe14}.
$F$ is finitely presented, as it can be described by $\langle A, B \ | \  [AB^{-1},A^{-1}BA], [AB^{-1},A^{-2}BA^2] \rangle$.
However, we often consider the equally well-known infinite presentation given by
$$
F=\langle x_0, x_1, \ldots \; | \; x_n x_k=x_k x_{n+1} \textrm{ for } k<n\rangle
$$
The (left) shift $\phi
: F\to F$ is the homomorphism of $F$ defined by $$\phi(x_i)=x_{i+1}, \quad i \in \N_0=\{ 0, 1, 2, 3, \ldots \} \ . $$

\subsubsection{Thompson's group $F$ as a group of fractions.}
The elements of $F$ admit a nice graphical description. 
Indeed, any element of $F$ can be described as an equivalence class of pairs of rooted planar binary trees with the same number of leaves.  .
For instance, the standard generator $x_0$ is represented by a pair of trees with three leaves.

We introduce this diagrammatic approach from a categorical point of view, as described in \cite{Jo16}.
See also \cite[Section 7.2]{Belk04}.
A  \emph{binary forest} with $n$ roots and $m$ leaves is an isotopy class of planar diagrams in the strip ${\R} \times [0,1]$, the roots and the leaves being respectively $n$ distinct points in ${\R} \times \{0\}$ and $m$ distinct points in ${\R} \times \{1\}$ joined by straight lines possibly bifurcating from the bottom to the top. We number the roots and
the leaves from left to right. 
A \emph{tree} is a binary forest with only one root. 
The composition $f \circ g$ (or simply $fg$) of two binary forests $f$ and $g$ is defined when the number of leaves of $g$ equals the number of roots of $f$. This is then the binary forest obtained by stacking vertically $f$ on the top of $g$ lining up the leaves of $g$ with the roots of $f$, followed by (vertical) rescaling. In this way, we obtain a category $\cF\equiv \cF_2$ with objects the natural numbers and morphisms between $n$ and $m$  the set ${\cF}(n,m)$ of binary forests with $n$ roots and $m$ leaves. Similarly, by using $k$-ary forests, one gets the category $\cF_k$.

It is useful to introduce a special notation for some selected trees: $\vert$ is the tree with only one leaf, while $Y$ denotes the tree with two leaves. We also denote by $\bullet$ the operation of horizontal concatenation of forests, so that if $f \in \cF(n,m)$ and $g \in \cF(n',m')$ then $f \bullet g \in \cF(n+n',m+m')$ ($f$ on the left of $g$), see the example below.
$$
\includegraphics[scale=0.25]{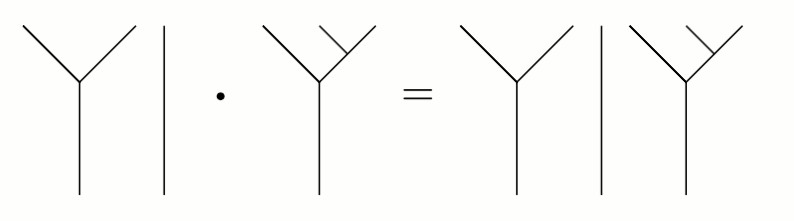}
$$

Note that the set of trees, denoted by $\fT$, is a directed set where $s\leq t$ if and only if there exists a forest $f$ such that $f\circ s=t$.

Consider the set of pairs of trees $(t,s)$, where $t$ and $s$ have the same number of leaves, that we mod out by the relation generated by $(t,s) \sim (f \circ t, f \circ s)$ for any composable forest $f$. 
We denote by $\frac{t}{s}$ or $(t,s)$ the class of the pair $(t,s)$. 
We define a multiplication on this quotient set by the formula 
$$\frac{t}{s} \cdot \frac{t'}{s'} = \frac{pt}{qs'}$$
where $ps = qt'$.
It gives a group structure such that the inverse of $\frac{t}{s}$ is $\frac{s}{t}$ and the neutral element is $\frac{t}{t}$ for any tree $t$.
In turn, the group thus obtained is called the \emph{group of fractions} of the category $\cF$ and it is isomorphic to $F$.
Graphically we represent an element $\frac{t}{s}$ by first drawing the tree $t$ upside down and then $s$ matching up the leaves of the two trees.
For instance the generators $x_0$ and $x_1$ 
are depicted below
$$
x_0=  			\begin{array}[c]{l}\includegraphics[scale= .15]{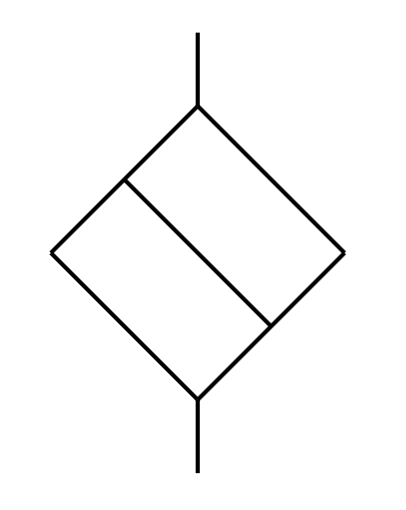}\end{array}\; ,\qquad x_1=  			\begin{array}[c]{l}\includegraphics[scale= .15]{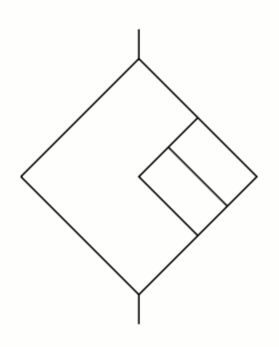}\end{array}\; . 
$$

In this diagrammatic approach, the shift $\phi$ is described by $$\phi\left(\frac{t}{s}\right) = \frac{(\vert \bullet t)\circ Y}{(\vert \bullet s)\circ Y} \ . $$
With this notation, for instance, $x_0=\frac{(Y \bullet \vert) \circ Y}{(\vert \bullet Y) \circ Y}$.

\subsubsection{Jones representations of Thomspon's group $F$.}
In \cite{Jo14}, Jones constructed a large class of unitary representations for $F$ and also for Thompson's group $T$.
Given a triple $(\cP,M,R)$ where $\CP$ is a \SPA\ (or even any non-degenerate C*-planar algebra), $M$ a module over the rectangular category of $\cP$, and $R$ a certain normalized element of $\cP$, we can construct a representation of $F$.
It $M$ is a module over the affine category of $\cP$, then we obtain a representation of the larger group $T$.
Jones generalized this construction in a very beautiful way in \cite{Jo16}.
Given a well behaved category $\CC$ he constructed a group of fraction $G_\CC$.
Then any functor $\Phi:\CC\to \Hilb$ with target the category of Hilbert spaces with isometries for morphisms provides a unitary representation of the group of fractions $G_\CC$ associated to $\CC$ \cite{Jo16}.

We are interested in representations of $F$ given by a triple $(\cP,M,R)$ where $\cP$ is a \SPA\ (most of time the $\TLJ$-\PA) and $M$ is the regular module over $\Rec(\cP)$.
We recall the construction of \cite{Jo14} in this particular case.
Fix a \SPA\ $\cP$ and $R\in \cP_2^+$.
We say that $R$ is normalized if we have the following identity
$$
\includegraphics[scale=0.125]{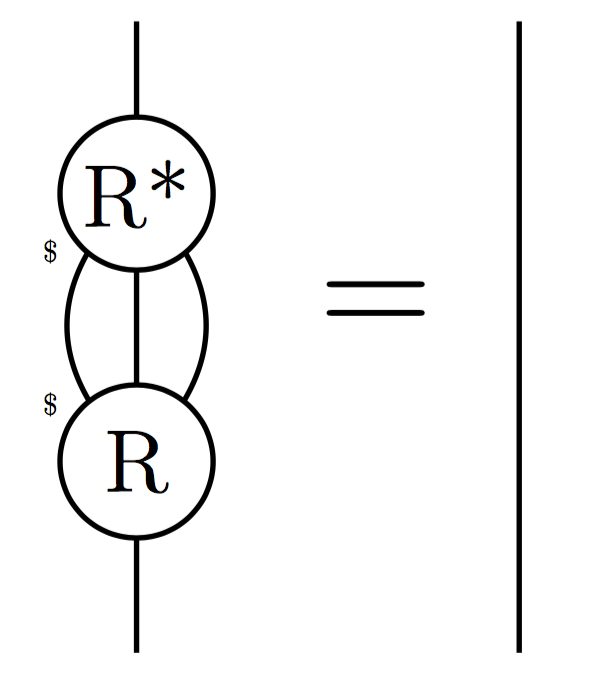}
$$
Consider the morphism
$$
\includegraphics[scale=0.125]{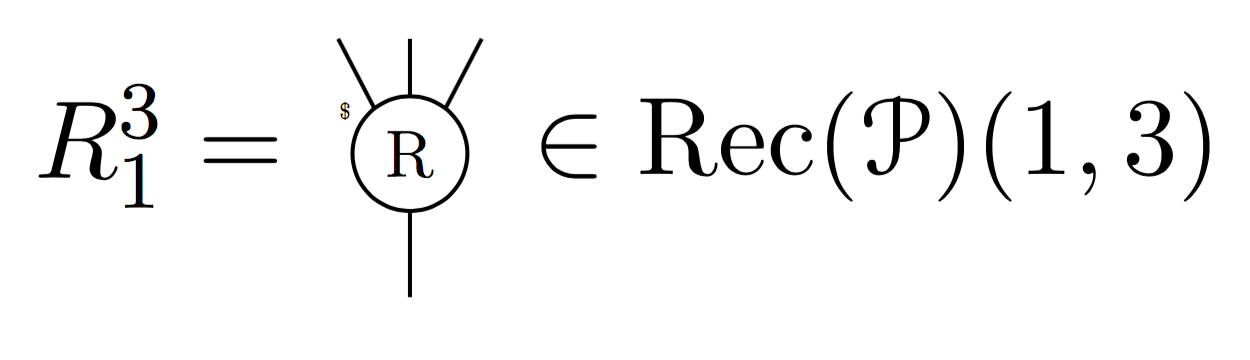}
$$ 
that we simply denote by $R$ if the context is clear.
Consider a binary tree $t$.
Let $\tilde t$ be the unique ternary tree which is obtained from $t$ by  replacing any binary branching by a ternary branching where the additional branch go straight to the top in the middle.
We then replace any branching of $\tilde t$ by an instance of $R_1^3$ which gives us an element $\Phi(t)\in \Rec(\cP)(1,\tilde m)$ where $\tilde m:=2m-1$.
For example, 
$$
\includegraphics[scale=0.2]{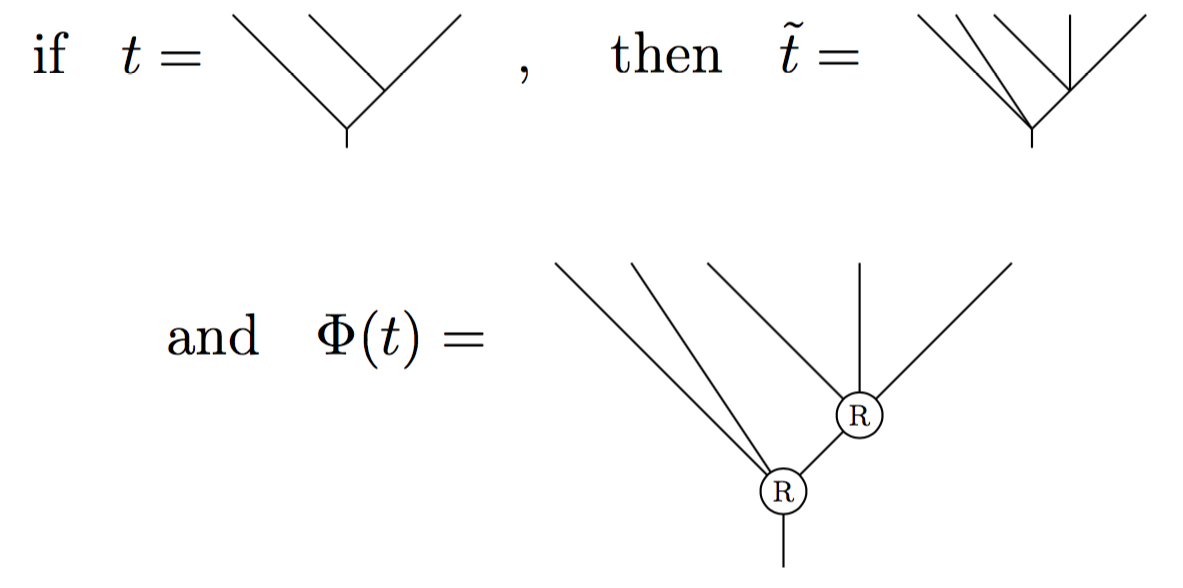}
$$ 
If $f$ is a binary forest with trees $t_1,\cdots,t_n$, then we construct the ternary forest $\tilde f = \tilde t_1 \bu \vert \bu \tilde t_2 \bu \vert \bu \cdots \vert\bu \tilde t_n$ where a trivial tree is added between $\tilde t_k$ and $\tilde t_{k+1}$ for any $1\leq k\leq n-1$.
Then replace any branching of $\tilde f$ by an instance of $R_1^3$.
If $f\in \cF(n,m)$, we obtain a morphism $\Phi(f)\in \Rec(\cP)(\tilde n , \tilde m)$ where $\tilde n:=2n-1$.
For any $n\geq 1$ let $\Phi(n)\subset \Rec(\cP)(1,n)$ be the span of $\Phi(t)$ where $t$ runs over all tree with $n$ leaves.
Given a tree $t$ with $n$ leaves we put $\fH_t:=\{(t,\xi) : \xi \in \Phi(n)\}$ that is a copy of $\Phi(n)$ indexed by $t$ that we equip with the restriction of the inner product of $\Rec(\cP)(1,n)$.
Note that when $\cP$ is irreducible we have that the inner product of $ \Rec(\cP)(1,n)$ is given by $\langle x, y \rangle := y^\dag\circ x$ where $\Rec(\cP)(1,1)$ is identified with $\C$.
Define the quotient space 
$$X:=\{ (t,\xi) : t \text{ a tree } , \xi\in \Phi(\target(t))\}/\sim$$
where $\sim$ is the equivalence relation generated by $(t,\xi)\sim (ft, \Phi(f)\circ \xi)$.
This quotient space is nothing but  the inductive limit of the directed system $(\fH_t)_{t\in\fT}$ with inclusion maps $\iota_t^{ft}:\fH_t\to \fH_{ft}, (t,\xi)\mapsto (ft, \Phi(f)\circ\xi)$.
Note that the map $\iota_t^{ft}$ is an isometry and thus we obtain an inner product on $X$ that makes it a pre-Hilbert space.
Let $\scrH$ be its completion and identify $\fH_t$ with a subspace of $\scrH$  for any tree $t$.  
We denote by $\frac{t}{\xi}$ or even $(t,\xi)$ the equivalence class of $(t,\xi)\in \fH_t$ inside $\scrH$.
We can now introduce the Jones representation 
$\pi:F\to \cU(\scrH)$ as the unitary representation densely defined by the formula
$$\pi\left(\frac{t}{s}\right) \frac{v}{\xi} := \frac{pt}{\Phi(q)\xi} \text{ where } p\circ s = q\circ v.$$
In particular, $\pi\left(\frac{t}{s}\right)\frac{s}{\xi} = \frac{t}{\xi}$.
We write $\Omega$  for the unit vector $(\vert,1)$ belonging to $\fH_\vert\subset\scrH$.
By construction, $\Omega$ is a cyclic vector inside $\scrH$.
To emphasize the role of $\cP$ and $R$ we will sometimes add the subscript $\cP, R$ or simply $R$ to $\Phi,\pi,\scrH,$ and $\Omega$.
Note that we slightly modified the definition of the representation of $\pi$ with respect to \cite{Jo14}.
Indeed,   
therein the space $\scrH$ is the completion of $\{(t,\xi) : t\in\fT, \xi\in\Rec(\cP)(1,\target(t))\}/\sim$ and thus $\Omega$ is not necessarily cyclic.
The reason is that we are mostly interested in the cyclic representation generated by $\Omega$.
We denote by $\varphi(\cdot)=\langle\pi(\cdot) \Omega , \Omega\rangle$ the vector state given by $\Omega$ and notice the following equality:
\begin{equation}
\varphi\left(\frac{t}{s}\right) = \langle \Phi(s) , \Phi(t)\rangle \text{ for any element } \frac{t}{s}\in F.
\end{equation}

We can interpret the representation $\pi$ in the setting of \cite{Jo16} where $\Phi$ defined a functor from the category of forests $\cF$ to the category of Hilbert spaces with isometries as morphisms. 
This functor provides a unitary representation $(\pi_\Phi,\scrH_\Phi)$ 
which 
is unitarily equivalent to the representation $(\pi,\scrH)$ described above.
 
As an illustration we provide an explicit computation of some values of $\varphi$ for a specific choice of $(\cP,R).$
Those values will be used in Section \ref{sec:TLJ}.

\begin{proposition}\label{phix0x1}
Consider a subfactor planar algebra $\cP$ with loop parameter $\delta\in\{2\cos(\pi/n) : n \geq 4\} \cup[2,\infty)$ and two complex numbers $a,b$ satisfying the equation $1= \delta(|a|^2 + |b|^2) + \bar a b + a\bar b.$
Put $R_1 = \vert \, \vert , \quad R_2= {}^\cup_\cap$, and $R:=aR_1+ bR_2$.
Then, $R$ is a normalized element of $\cP$ that defines a functor $\Phi$ and a unitary representation $(\pi,\scrH)$ with vacuum vector $\Omega$ and associated vector state $\varphi.$

Furthermore, evaluating $\varphi$ on $x_0, x_0 x_1, x_1 x_0 \in F$
we find the following equalities:
\begin{align*}
\varphi(x_0) & = 1 + | ab |^2 (1-\delta^2) \\
\varphi(x_1x_0) & = \varphi(x_0)^2 \\
\varphi (x_0x_1 ) & = \varphi(x_1x_0) + | a^2 b |^2 (\delta^2 - 1)^2 \left( \frac{\delta}{\delta^2 - 1} - |b|^2 \right).
\end{align*}
\end{proposition}

\begin{proof}
According to our setting, we will repeatedly use the following rules:
\begin{align} \label{formula-semplificazioni}
& \begin{array}[c]{l}\includegraphics[scale= .2]{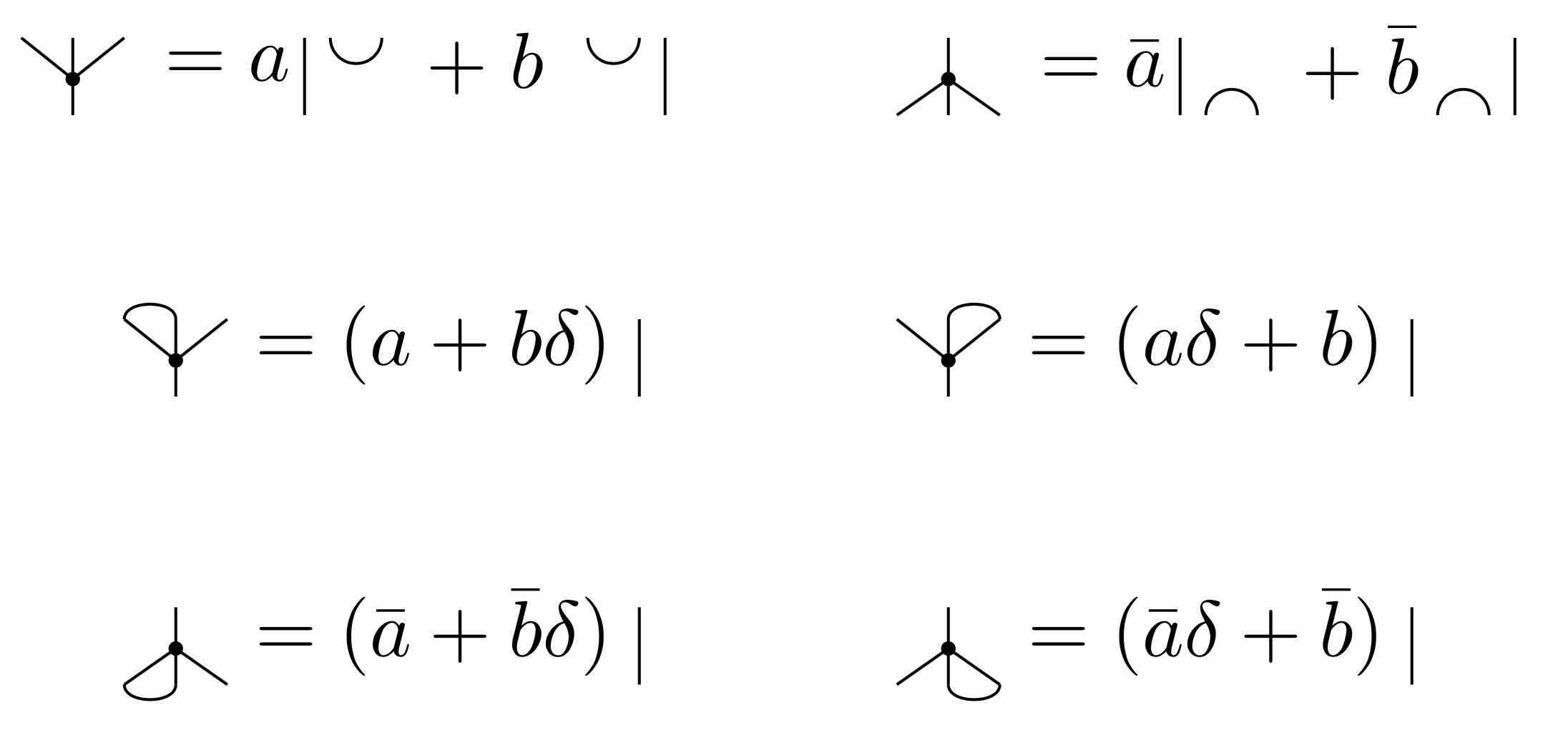}\end{array}
\end{align}
(to ease the graphical description of the elements of the planar algebra we replace the occurrences of $R$  by a small black disk).
It is then easy to check that the normalization condition reads as $1= \delta(|a|^2 + |b|^2) + \bar a b + a\bar b$, which is our assumption.
\\
Since 
$x_0=  			\begin{array}[c]{l}\includegraphics[scale= .1]{fig15}\end{array}$, we compute
\begin{align*}
\varphi(x_0) \; | & = \begin{array}[c]{l}\includegraphics[scale= .25]{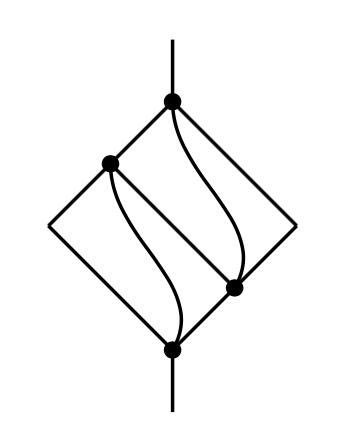}\end{array} \\
& = a\begin{array}[c]{l}\includegraphics[scale= .2]{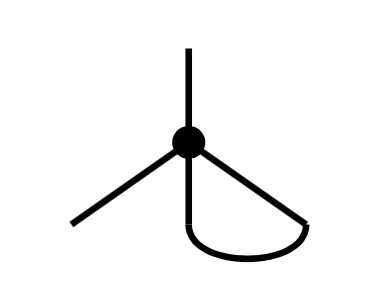}\end{array} +b \left( \bar a \;\begin{array}[c]{l}\includegraphics[scale= .3]{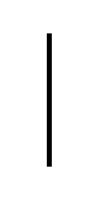}\end{array} +\bar b (	
a+b\delta ) \begin{array}[c]{l}\includegraphics[scale= .2]{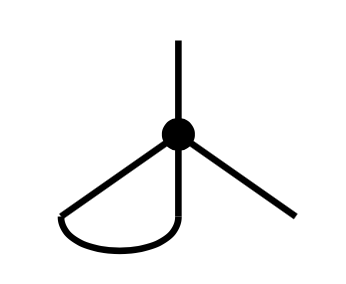}\end{array} 
						\right)\\
& = \left[ |a|^2 \delta + a\bar b +b\bar a + |b|^2	(a+b\delta)(\bar a +\bar b \delta)\right] |\\					
& = \left[  1+ |ab|^2(1-\delta^2)\right]  \; |
\end{align*}
where we used the normalization condition.\\
Now we consider
$$
x_0x_1=  			\begin{array}[c]{l}\includegraphics[scale= .25]{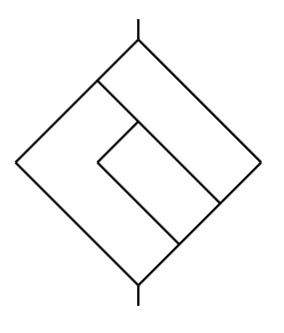}\end{array}\; ,\qquad x_1x_0=  			\begin{array}[c]{l}\includegraphics[scale= .25]{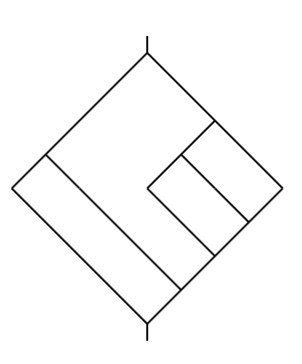}\end{array}\; . 
$$
On the one hand, we have
$$
\varphi(x_1x_0)\; | = \begin{array}[c]{l}\includegraphics[scale= .325]{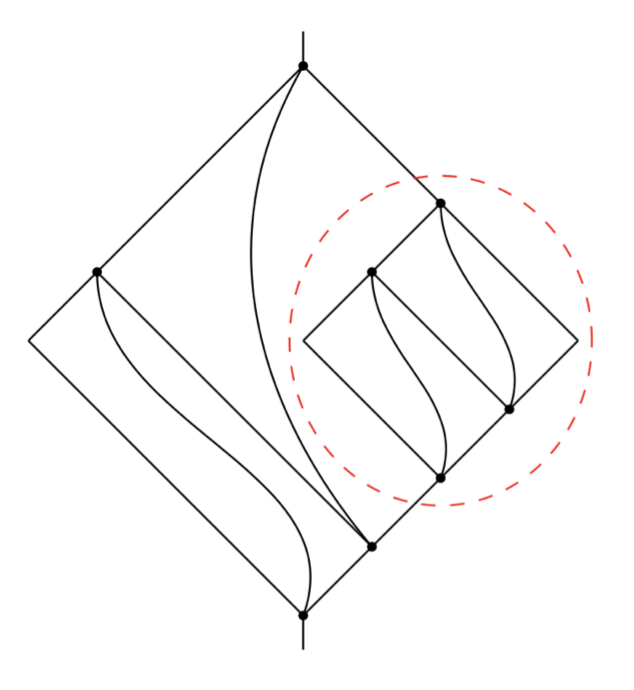}\end{array}	= \varphi(x_0) \begin{array}[c]{l}\includegraphics[scale= .325]{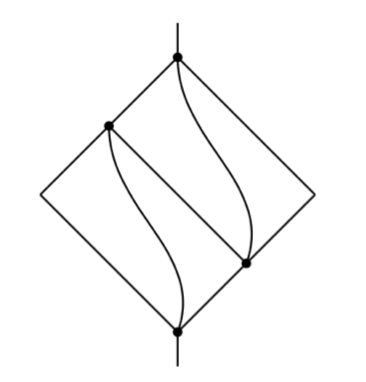}\end{array} 
= (\varphi(x_0))^2 \; | \; ,
$$
on the other hand,
\begin{align*}
\varphi(x_0x_1)\; | & = \begin{array}[c]{l}\includegraphics[scale= .25]{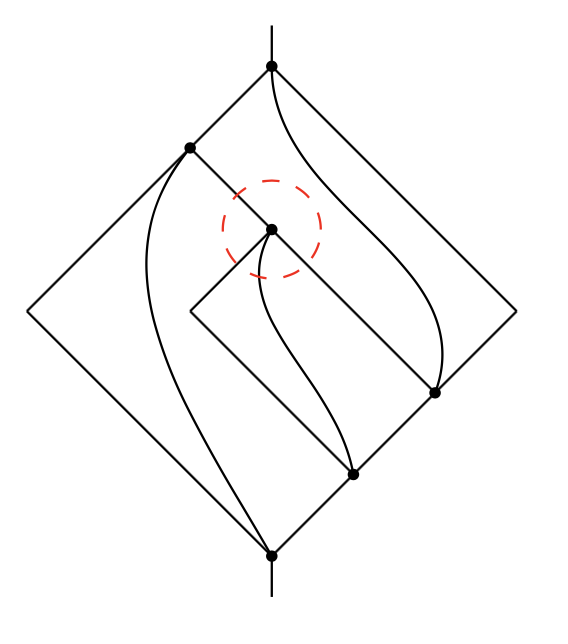}\end{array} = \bar a \begin{array}[c]{l}\includegraphics[scale= .25]{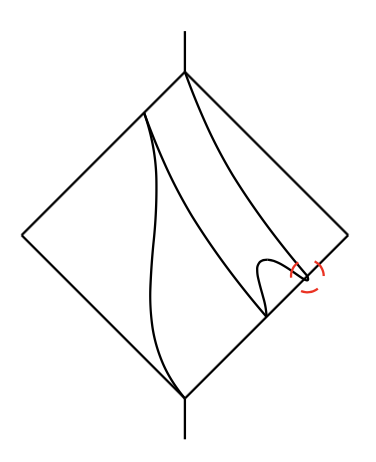}\end{array} +\bar b \begin{array}[c]{l}\includegraphics[scale= .25]{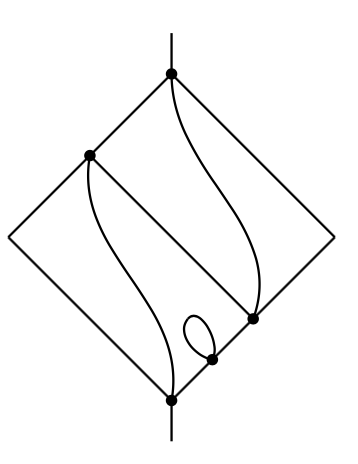}\end{array} \\
& = \bar a \left( a  \begin{array}[c]{l}\includegraphics[scale= .25]{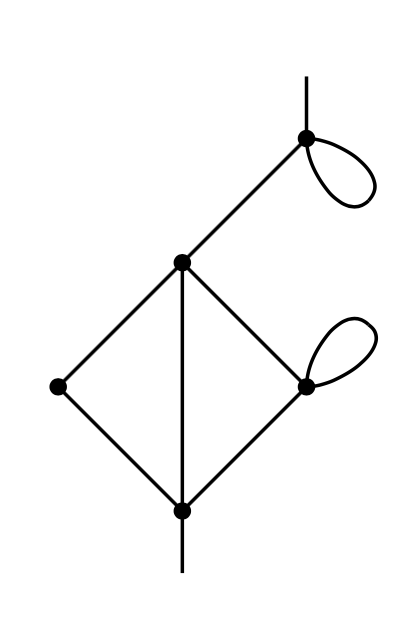}\end{array}  +b  \begin{array}[c]{l}\includegraphics[scale= .25]{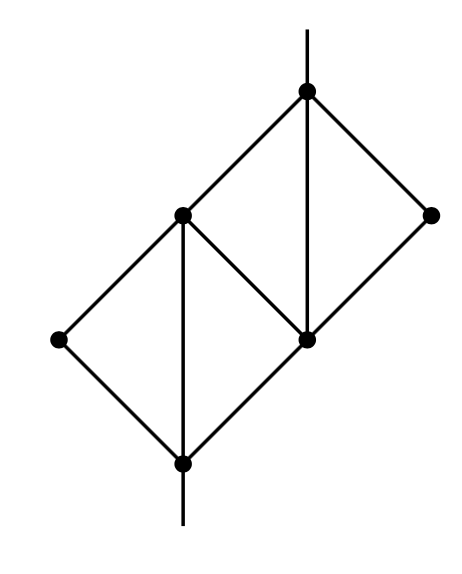}\end{array} \right)
+ \bar b (a+b\delta)  \begin{array}[c]{l}\includegraphics[scale= .25]{fig14}\end{array} \\
& = |a|^2 (a\delta +b) (\bar a \delta+\bar b)+\bar a b  \begin{array}[c]{l}\includegraphics[scale= .25]{fig14}\end{array}  +\bar b (a+b\delta)  \begin{array}[c]{l}\includegraphics[scale= .25]{fig14}\end{array}\\
& = |a|^2 (a\delta +b) (\bar a \delta+\bar b)+(1-|a|^2 \delta)  \begin{array}[c]{l}\includegraphics[scale= .25]{fig14}\end{array}\\
& =  \left[ |a|^2 (a\delta +b) (\bar a \delta+\bar b)+(1-|a|^2 \delta)\varphi(x_0) \right]  \; | \\
& = \left[  |a|^2 |a\delta +b|^2+(1-|a|^2 \delta)(1+|ab|^2(1-\delta^2))\right]  \; |
\end{align*} 
The conclusion now follows by lengthy but straightforward computations, using again the normalization condition.
\end{proof}
 
\section{First results and observations on the Jones representations}\label{sec-First-results}
\subsection{Equivalent representations}
In this section $\cP$ is a subfactor planar algebra and $R$ is a normalized element of $\cP.$
We prove that the representation we get does not change up to unitary equivalence if we multiply $R$ by a phase.

\begin{proposition}\label{phase}
Consider the Jones representation $(\pi_R,\scrH_R)$ associated to a couple $(\cP,R).$
If $z$ is a complex number of modulus one, then $zR$ is still normalized in $\cP$ providing a representation $\pi_{zR}$ that is unitary equivalent to $\pi_R.$
\end{proposition}

\begin{proof}
For any tree $t$ we write $\fH_{R,t}$ and $\fH_{zR,t}$ the $t$-component subspace of $\scrH_R$ and $\scrH_{zR}$ respectively.
Observe that $\Phi_R(n)=\Phi_{zR}(n)$ for any $n\geq 1$ and thus we can identify $\fH_{R,t}$ and $\fH_{zR,t}$ for any tree $t$.
Given a tree $t$ with $n$ leaves we put $U_t:\fH_{R,t}\to \fH_{zR,t}$ that is the multiplication by $z^{n-1}.$
Observe that $\Phi_{zR}(f) \circ U_t = U_{ft} \circ \Phi_R(f)$ for any forest $f$ implying that the family of maps $(U_t)_{t\in\fT}$ defines a map $U:\scrH_R\to \scrH_{zR}$.
It is easy to see that $U$ is a unitary transformation which intertwines $\pi_R$ and $\pi_{zR}.$
\end{proof}

The next proposition points out that we can apply a planar algebra automorphism to $R$ without changing the unitary class of the representation.

\begin{proposition}
Consider a couple $(\cP,R)$ and $\alpha\in\Aut(\cP)$ an automorphism of the planar algebra $\cP$.
Then the representations $\pi_R$ and $\pi_{\alpha(R)}$ are unitarily equivalent.
\end{proposition}

\begin{proof}
Note that $S:=\alpha(R)$ is normalized since $\alpha$ commutes with the action of the tangles and thus defines a representation $(\pi_S,\scrH_S).$
Denote by $\fH_{R,t}$ and $\fH_{S,t}$ the $t$-component subspaces of $\scrH_R$ and $\scrH_S$ respectively for each tree $t$.
Consider the map 
$$U_t: \fH_{R,t}\to \fH_{S,t}, (t,\xi) \mapsto (t,\alpha(\xi)),$$
where we identify $\fH_{R,t}$ and $\fH_{S,t}$ with subspaces of the planar algebra $\cP$.
The system of maps $(U_t)_{t\in\fT}$ agrees with the inductive structures of $(\fH_{R,t})_{t\in\fT}$ and $(\fH_{S,t})_{t\in\fT}$ and thus defines a map $U:\scrH_R\to\scrH_S.$
It turns out that $U$ is a unitary transformation satisfying $\Ad(U)\circ \pi_R = \pi_S.$
\end{proof}

Note that the TLJ-planar algebra does not have any non-trivial automorphisms.
However, it has some symmetries that behave almost as automorphisms.
Fix a loop parameter $\delta$ and consider the TLJ-planar algebra $\TLJ:=\TLJ(\delta)$.
Let $D(m,n)$ be the set of rectangular TLJ-diagrams with $m$ boundary points on the bottom and $n$ on top.
The map that consists of a symmetry with respect to a vertical line on $D(m,n)$ 
(that we extend linearly to $\Rec(\TLJ)(m,n)$) 
 induces an involution $\sigma_{\TLJ}:\Rec(\TLJ)\to \Rec(\TLJ)$ defined on the rectangular category associated to the planar algebra $\TLJ.$

\medskip
We also need the homeomorphism $x\mapsto 1-x$ of the unit interval  and the induced order two automorphism $\sigma_F$ of the Thompson's group $F$ given by the formula $\sigma_F(g)(x)=1-g(1-x)$ for $x\in [0,1]$ and $g\in F$ (when $F$ is viewed as a subgroup of the homeomorphism group of $[0,1]$).
If $\sigma$ is the involution on the set of forests that consists of a symmetry with respect to a vertical line, then 
$$\sigma_F\Big(\frac{s}{t}\Big) = \frac{ \sigma(s)}{\sigma(t)}.$$
For instance, it is not difficult to see that $\sigma_F(x_0)=x_0^{-1}$.
The automorphism $\sigma_F$ is not inner, see for instance \cite[Theorem 1]{Br96} for a thorough analysis of ${\rm Aut}(F)$
and also \cite[Section 4]{Heo}, where it is shown that $\sigma_F$ is not inner in the group von Neumann algebra $L(F)$.
Consider a normalized element $R\in\TLJ$.

\begin{proposition}\label{abvsba}
The map $\sigma_{\TLJ}$ induces a unitary transformation $u_\sigma: \scrH_R\to \scrH_{\sigma_{\TLJ}(R)}$ given by $(t,\xi)\mapsto (\sigma(t),\sigma_{\TLJ}(\xi))$, mapping the vacuum vector $\Omega_R$ into the vacuum vector $\Omega_{\sigma_{\TLJ}(R)}$ and such that
$$\Ad(u_\sigma)\circ \pi_R  = \pi_{\sigma_{\TLJ}(R)} \circ \sigma_F.$$
\end{proposition}

\begin{proof}
One can check that $\Phi_{\sigma_{\TLJ}(R)}(\sigma(f))=\sigma_{\TLJ}(\Phi_R(f))$ for any forest $f$.
Moreover, the maps $\sigma$ and $\sigma_{\TLJ}$ are multiplicative with respect to vertical concatenation. 
This implies that $u_\sigma$ is well defined.
It is easy to see that $u_\sigma$ is a unitary transformation whose inverse 
maps $(s,\eta)\in \scrH_{\sigma_{\TLJ}(R)}$ to $(\sigma(s), \sigma_{\TLJ}(\eta))\in \scrH_R.$ 
Consider a vector $(t,\xi)\in \fH_{\sigma_{\TLJ}(R),t}$ and an element $g=\frac{s}{r}\in F$. 
We have that
\begin{align*}
\Ad(u_\sigma)\circ \pi_R(g)  \frac{t}{\xi} & = u_\sigma \pi_R(g) \frac{\sigma(t)}{ \sigma_{\TLJ}(\xi)} =  u_\sigma \pi_R\left(\frac{ps}{ pr}\right) \frac{q\sigma(t)}{ \Phi_R(q)\sigma_{\TLJ}(\xi)}\\
& = u_\sigma \left(\frac{ps}{ \Phi_R(q)\sigma_{\TLJ}(\xi)}\right)= \frac{\sigma(ps)}{  \sigma_{\TLJ}\circ\Phi_R(q)\xi}= \frac{\sigma(ps)}{ \Phi_{\sigma_{\TLJ} (R) }(\sigma(q))\xi}
\end{align*}
where $pr=q\sigma(t)$.
On the other hand, we have that
\begin{align*}
 \pi_{\sigma_{\TLJ} (R)}(\sigma_F(g)) \frac{t}{\xi} & =\pi_{\sigma_{\TLJ}(R)}\left(\frac{\sigma(s)}{ \sigma(r) } \right)\frac{t}{\xi}
  =\pi_{\sigma_{\TLJ}(R)}\left(\frac{\sigma(p) \sigma(s) }{ \sigma(p) \sigma(r) }\right) \frac{ \sigma(q) t }{ \Phi_{\sigma_{\TLJ}(R)}(\sigma(q))\xi }\\
 & = \frac{\sigma(ps)}{ \Phi_{\sigma_{\TLJ} (R)}(\sigma(q))\xi}.
\end{align*}
This concludes the proof by a density argument.
\end{proof}

We can provide a similar statement for any subfactor planar algebra, but with an anti-unitary.
The proof is similar to the one given above and we leave it to the reader.

\begin{proposition}
Let $\CP$ be a subfactor planar algebra, $R$ a normalized element, with its associated functor $\Phi_R$ and unitary representation $(\pi_R,H_R).$
Define the following anti-linear involution of the rectangular category:
$$
\begin{tikzpicture}[baseline = -.1cm]
\draw (-0.25,-0.25)--(.25,-0.25);
\draw (-0.25,0.25)--(.25,0.25);
\draw (0.25,-0.25)--(.25,0.25);
\draw (-0.25,-0.25)--(-.25,0.25);
\draw (0,0.25)--(0,0.5);
\draw (0,-0.25)--(0,-0.5);
\node at (-1,0) {$J :$};
\node at (0,0) {$x$};
\node at (-0.45,0) {{\scriptsize $\$$}}; 
\node at (0,0.65) {{\scriptsize $n$}};  
\node at (0,-0.65) {{\scriptsize $m$}};
\end{tikzpicture}\mapsto
\begin{tikzpicture}[baseline = -.1cm]
\draw (-0.25,-0.25)--(.25,-0.25);
\draw (-0.25,0.25)--(.25,0.25);
\draw (0.25,-0.25)--(.25,0.25);
\draw (-0.25,-0.25)--(-.25,0.25);
\draw (0,0.25)--(0,0.5);
\draw (0,-0.25)--(0,-0.5);
\node at (0,0) {$x^*$};
\node at (0.45,0) {{\scriptsize $\$$}}; 
\node at (0,0.65) {{\scriptsize $n$}};  
\node at (0,-0.65) {{\scriptsize $m$}};
\end{tikzpicture}
$$
and put $S:=J(R)$.
Then the map $J$ is multiplicative, 
 $J(\Phi_R(f)) = \Phi_{J(R)}(\sigma(f))$ for any forest $f$, and the element $S$ is normalized. 
Moreover, $J$ induces an anti-unitary transformation $U_J:\scrH_R\to \scrH_S$ densely defined as $U_J((t,\xi)):=(\sigma(t), J(\xi))$ for any $(t,\xi)\in \fH_{R,t}$.
Finally, we obtain the following equality
$$\Ad(U_J)\circ \pi_R = \pi_S\circ \sigma_F.$$ 
\end{proposition}

In general, it would also make sense to restrict a representation $\pi_R$ of $F$ as above to some notable subgroups of $F$.
\begin{remark}
The subgroup $F^\sigma:=\{g\in F : \sigma_F(g)=g\}$ is isomorphic to $F$ itself since it corresponds to the maps $g:[0,1]\to [0,1]$ of $F$ whose graph is symmetric with respect to the point $(1/2,1/2).$ 
We have the following isomorphism: 
$$\alpha:F\to F^\sigma,\quad \frac{s}{t}\mapsto \frac{(s\bullet \sigma(s)) \circ Y }{ (t\bullet \sigma(t)) \circ Y }.$$
Consider the unitary representation $(\pi,\scrH)$ associated to a normalized element $R$ of an \emph{irreducible} subfactor planar algebra $\CP$.
We define a new unitary representation of $F$ as follows: $\pi^\sigma:=\pi\circ\alpha.$
Consider the vacuum vector $\Omega\in \scrH$ and observe that 
$$\langle\pi^\sigma(g)\Omega , \Omega \rangle = \langle\pi(g)\Omega , \Omega \rangle\langle\pi(\sigma_F(g))\Omega , \Omega \rangle, \quad g\in F.$$
Note that we use the fact that the 1-box space of $\CP$ is one-dimensional.
It follows from \cite[Proposition 2.4.1., p.~37]{Dix}  that 
$$\pi^\sigma \simeq \pi\otimes (\pi\circ\sigma_F).$$
\end{remark}

\subsection{Some properties of the Jones representations}
The shift endomorphism of the Thompson's group $F$ is the map $\phi_F$ defined as $\phi_F(x_i)=x_{i+1}$ for any $i\geq 0$ for the usual presentation of $F$, see Section \ref{sec:JonesRep}.
It can be diagrammatically defined as follows: if $t$ is a tree with $n$ leaves we consider $\phi(t)\in\cF(1,n+1)$ that is a new tree equal to the composition of the unique tree with two leaves $Y$ with the forest with the trivial tree on the left and the tree $t$ on the right, i.e.~$\phi(t)= ( \vert \bullet t)\circ Y$.
If $g\in F$ is described by the pair of trees $(t,s)$, then we observe that $\phi_F(g)$ is described by the pair of shifted trees $(\phi(t),\phi(s)).$
We drop the subscript $F$ for $\phi$ when it is clear from the context.

\begin{proposition}\label{prop:non-regular}
If $g\in F$ is non-trivial, then the sequence $g_0=g,g_{n+1}:=\phi(g_n), n\geq 0$ tends to infinity in $F$ 
(for the Fr\'echet filter). 
Consider a Jones representation $\pi$ constructed with an irreducible $\cP$ subfactor planar algebra (
in particular, the one box space $\cP_1^+$ is one dimensional), then the vacuum vector state $\varphi$ is invariant under the shift of $F$.
In particular, $\pi$ is not contained in any direct sum of copies of the left regular representation as long as there exists some non-trivial $h \in F$ with $\varphi(h) \neq 0$.
\end{proposition}

\begin{proof}
If $g$ is a non-trivial element, then it is described by a reduced pair of different trees.
It is then easy to see that $\phi(g)$ is still described by a reduced pair of trees but with one more leaf.
This implies that the number of leaves of the reduced pair associated with $g_n:=\phi^n(g)$ tends to infinity and thus $g_n$ tends to infinity.

The invariance of the vacuum vector state readily follows from the normalization property of $R$ and the irreducibility condition on $\cP$.
Indeed, identify the one box space $\cP_1^+$ with $\C$ via the trace and consider an element $g=\frac{t}{s}\in F.$
We have that 
$$\Phi(t)^*\circ \Phi(s) =  \begin{array}[c]{l}\includegraphics[scale= .325]{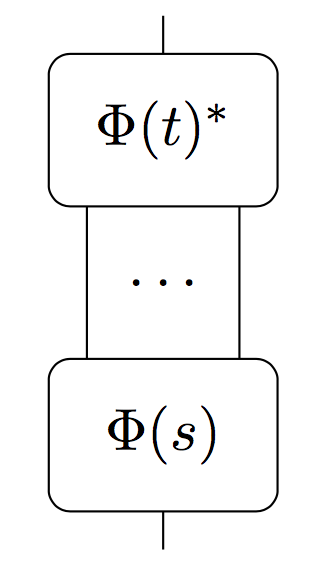}\end{array} = \varphi(g) \; |$$
Then 
$$\varphi(\phi(g)) \; \vert = \begin{array}[c]{l}\includegraphics[scale= .075]{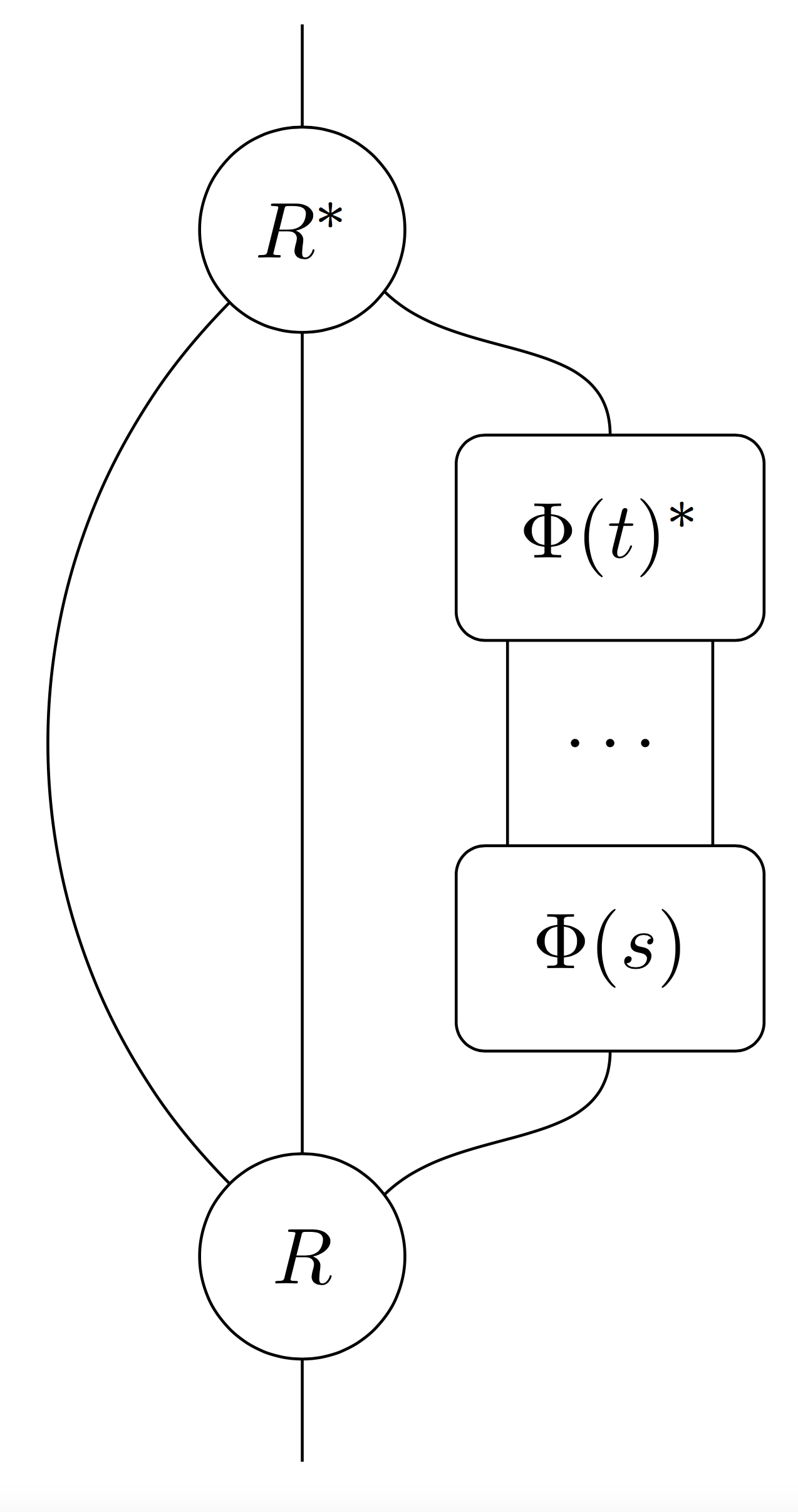}\end{array} = \varphi(g) \begin{array}[c]{l}\includegraphics[scale= .1]{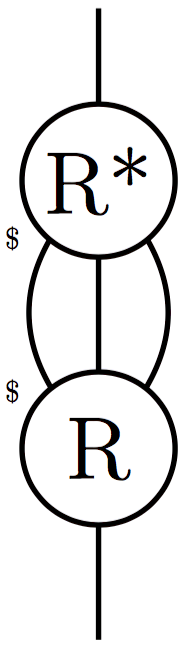}\end{array}  = \varphi(g) \; |$$

Suppose that there exists a non-trivial $h\in F$ such that $\varphi(h)\neq 0.$
Then, the vacuum vector state does not tend to zero at infinity  since $\varphi(\phi^n(h)) = \varphi(h)\neq 0$.
Recall that any vector in the carrier Hilbert space of $\lambda^{\oplus \infty}$, the infinite direct sum of copies of the left regular representation of $F$, tends to zero at infinity.
Therefore, $\pi$ is not contained in $\lambda^{\oplus \infty}$.
\end{proof}

In particular,  it follows that a Jones representation with a non-trivial stabilizer of the vacuum vector is never contained in the regular representation of $F$.
Note that such a $\pi$ can still be unitarily equivalent to a quasi-regular representation, as it was shown in \cite{Jo14}, cf.~Section \ref{sec:stab}.
In Section \ref{sec:notzero} we will show that for any unitary representation constructed with TLJ and two real parameters $a,b$ we have   $\varphi(x_0)\neq 0$ for $x_0\in F$ in all the cases but one.

\section{Family of representations constructed with the TLJ planar algebra}\label{sec:TLJ} \label{sec-TLS-reps}
\subsection{First observations}

Fix a loop parameter $\delta\in \{2 \cos(\pi/k) \ | k = 4, 5, 6, \ldots \} \cup [2,+\infty)$ and consider the Temperley-Lieb-Jones planar algebra $\TLJ:=\TLJ(\delta)$ with this parameter.
A normalized element $R$ of $\TLJ$ is a linear combination of the identity tangle
$R_1 = \vert \, \vert $
and  a multiple of the Jones projection $R_2= {}^\cup_\cap$\, .

Observe that if $R=aR_1+bR_2\in \TLJ$, where $a,b \in \C$, then $R $ is normalized if and only if 
\begin{equation}\label{normaliz}
(a,b)\in \CC^\delta :=\{(x,y)\in\C\times\C\; | \; \delta(|x|^2+|y|^2)+x\bar{y}+\bar{x}y=1\} \ . 
\end{equation}
Note that the curve $\CC^\delta$ is invariant under various operations such as the flip of coordinates, the simultaneous complex conjugation of both the coordinates and the simultaneous multiplication of both the coordinates by a phase factor.
Given $(a,b)\in\CC^\delta$, we write $\pi^\delta_{a,b}$ and $\varphi^\delta_{a,b}$ for the associated Jones representation and vacuum vector state.

\smallskip
We first observe that if $a$ or $b$ is equal to zero, then we obtain the trivial representation.
The proof is rather obvious but we include it in order to illustrate the formalism. 

\begin{proposition}\label{prop:pitrivial}
Consider $(a,b)\in \CC^\delta$ such that $a$ or $b$ is equal to zero and let $(\pi,\scrH)$ be the associated representation.
Then $\scrH$ is one dimensional and $\pi$ is the trivial representation.
\end{proposition}

\begin{proof}
Assume that $b=0.$
Up to a multiplication by a phase we can assume that $a=\delta^{-1/2}.$
Write $\Phi_R$ the functor associated with $R$.
We claim that $\Phi_R(t)= \delta^{-(n-1)/2} \vert \bu \cup^{\bu n-1}$ for any tree $t$ with $n$ leaves.
The proof can be done by induction. The cases $n=1$ and $2$ are trivial. 
Suppose that the claim is true for $n-1$ and let $t$ be a tree with with $n$ leaves. 
Then $t$ may be seen as a composition of a tree $t'	\in \cF_2(1, n-1)$ and a forest $f\in \cF_2(n-1,n)$, with 
$$f=\begin{tikzpicture}[baseline = -.1cm]
\draw (-.5,-.5)--(-.5,.5);
\node at (-.08,.35) {\ldots};
\draw (.25,-.5)--(.25,.5);
\draw (.75 , .25)--(.5,.5);
\draw (.75,.25)--(1,.5);
\draw (1.25,-.5)--(1.25,.5);
\draw (0.75,-.5)--(0.75,.25);
\node at (1.75,.35) {\ldots};
\draw (2.1,-.5)--(2.1,.5);
\end{tikzpicture}
$$
where there are $k$ vertical edges on the left and $n-k-2$ vertical edges on the right. 
Consider the transformation $\thicksim : \cF_2\to \cF_3$ defined by replacing any binary branching with a ternary branching and adding a vertical straight line between two consecutive roots. Thus  we see that $\widetilde{t'}	\in \cF_3(1, 2n-3)$ and  $\tilde{f}\in \cF_3(2n-3,2n-1)$, where $\tilde{f}$ has the form
$$\tilde{f}=\begin{tikzpicture}[baseline = -.1cm]
\draw (-.5,-.5)--(-.5,.5);
\node at (-.08,.35) {\ldots};
\draw (.25,-.5)--(.25,.5);
\draw (1.25,-.5)--(1.25,.5);
\draw (0.75,-.5)--(0.75,.5);
\draw (.75 , .25)--(.5,.5);
\draw (.75,.25)--(1,.5);
\node at (1.75,.35) {\ldots};
\draw (2.1,-.5)--(2.1,.5);
\end{tikzpicture}
$$
with $2k$ straight lines on the left and $2n-2k-4$ on the right. Now, 
$$\Phi_R(\tilde{f})=\begin{tikzpicture}[baseline = -.1cm]
\draw (-.5,-.5)--(-.5,.5);
\node at (-.08,.35) {\ldots};
\draw (.25,-.5)--(.25,.5);
\draw (1,.5) arc (0 : -180 : .25) ;
\draw (1.25,-.5)--(1.25,.5);
\node at (1.75,.35) {\ldots};
\draw (2.1,-.5)--(2.1,.5);
\end{tikzpicture}
$$
is a morphism
with $2k+1$ straight lines on the left and $2n-2k-4$ on the right. 
By the induction hypothesis $\Phi_R(t')$ is a straight vertical line followed by $n-2$ cups. 
Since there is an odd number of straight lines to the left of the cup of $\Phi_R(\tilde{f})$ we obtain that $\Phi_R(t)=\Phi_R(\widetilde{t'})\circ \Phi_R(\tilde{f})=\delta^{-(n-1)/2} \vert \bu \cup^{\bu n-1}$, which proves the claim.

It follows that each space $\fH_t$ is one dimensional, as well as the inductive limit $\scrH.$
Now consider a group element $g=(t,s)\in F$ described by a pair of trees with $n$ leaves and observe that $\varphi(g)=\langle \Phi(s)\Omega, \Phi(t)\Omega\rangle = \langle \delta^{-(n-1)/2} \vert \bu \cup^{\bu n-1} , \delta^{-(n-1)/2} \vert \bu \cup^{\bu n-1}\rangle = 1$.
Therefore, $\pi$ is the trivial representation.\\
The case $a=0$ can be handled in a similar way but it also follows at once from Proposition \ref{exchab}.
\end{proof}
The following result is an immediate consequence of Proposition \ref{abvsba}.
\begin{proposition} \label{exchab}
For each $\delta$, we have $\pi^\delta_{a,b} \circ \sigma_F \cong \pi^\delta_{b,a}$ and
$\varphi^\delta_{a,b} \circ \sigma_F = \varphi^\delta_{b,a}$  for all $(a,b) \in \CC^\delta$.
\end{proposition}
In particular, for
 any   $\delta$ as above, there exists only one class $\pi^\delta_+:=\pi^\delta_{a,a}$ of unitary representations of $F$ with $a=b$, that can be obtained by taking $a= 1/\big(2(\delta + 1)\big)^{1/2}$. 
By symmetry, the class of $\pi^\delta_+$ is not changed if we compose it with  $\sigma_F$.
Likewise, there exists only one class $\pi^\delta_-$ with $a=-b$,  that can be obtained by taking $a= 1/\big(2(\delta - 1)\big)^{1/2}$, again still invariant under $\sigma_F$ up to unitary equivalence.
For the time being, it is unclear whether $\pi^\delta_+$ and $\pi^\delta_-$ are unitarily equivalent.

\medskip
As one can expect, apart from the trivial representation, all the other representations have trivial kernel.
\begin{proposition}
The representation $\pi_{a,b}^\delta$ is faithful for any $(a,b)\in \CC^\delta$ with $a\neq 0\neq b$.
\end{proposition}
\begin{proof}
If $\pi_{a,b}^\delta$ had a non-trivial kernel then, by \cite[Theorem 4.3]{CFP96}, the quotient group by this kernel would necessarily be abelian and thus $\varphi(xy)=\varphi(yx)$ for all $x,y\in F.$ 
However, thanks to  Proposition \ref{phix0x1}, the equality $\varphi(x_0 x_1) = \varphi(x_1 x_0)$ leads to
$|b|^2 = \frac{\delta}{\delta^2 -1}$ and, similarly, $\varphi \big(\sigma_F(x_0 x_1)\big) = \varphi\big(\sigma_F(x_1 x_0)\big)$ leads to 
$|a|^2 = \frac{\delta}{\delta^2 -1}$, which are easily seen to be incompatible with the normalization condition and the fact that $\delta>1.$
\end{proof}

\subsection{Matrix coefficients which tend to zero}
Throughout the rest of this section we consider $\delta\in \{2 \cos(\pi/k) \ | \ k = 4, 5, 6, \ldots \} \cup [2,+\infty)$ as above and a {\it real} couple $(a,b)\in\CC^\delta$ such that $a\neq 0\neq b$.
Note that $\CC_\R^\delta:=\CC^\delta \cap \R^2$ is an ellipse whose axes are rotated by $\pi/4$ in the plane, with semi-major and semi-minor axes given by $1/\sqrt{\delta-1}$ and $1/ \sqrt{\delta + 1}$, respectively.
Set $\CP:=\TLJ(\delta)$ the TLJ-planar algebra with loop parameter $\delta$ and consider $R:=R_{a,b}$ that is a normalized element of $\cP_1$.

We will show that there exists a sequence $\ga_n\in F$ such that $\lim_n\varphi(\ga_n)=0$ 
when $a$, $b$ are real.
In fact, the sequence $(\ga_n)_n$ does not depend on the choice of the parameters $(\delta,a,b)$.

Consider the pair of trees $s_n,t_n\in \cF(1,n)$ for $n\geq 2$ defined as follows.
Let $s_2$ be the unique tree with one root and two leaves and  $s_{n+1}=f_{n}\circ s_n$ where $f_n$ is the unique forest with $n$ roots and $n+1$ leaves such that its $n$-th root is connected to its $n$-th and $n+1$-th leaves.
We define $t_n$ as the reflection of $s_n$ with respect to a vertical axis.
For example,
$$s_4=\DiagS \ ,  \text{ } t_4 = \DiagT.$$
Consider the Thompson's group element $\ga_n:=\frac{s_n}{t_n}$.
Note that $\ga_n=x_0^{n-1}$ for the usual presentation of Thompson's group $F$, see Section \ref{sec:JonesRep}.

\begin{proposition}\label{prop:coefzero}
We have that $\lim_n\varphi(\ga_n)=0$.
\end{proposition}

\begin{proof}
Since the TLJ-planar algebra is irreducible, we can identify the algebra of (1,1)-rectangular labelled tangles $\Rec(\CP)(1,1)$ with the algebra of complex numbers.
In particular, we obtain that $\varphi(\ga_n)=\Phi(s_n)^*\circ \Phi(t_n)$, that we denote by $d_n$.
We introduce a sequence of tangles $e_n\in \Rec(\TLJ(\delta))(1,1)\simeq \C$ that is derived from the sequence $(d_n)_n$ such that 
$$e_2=\Diagetwo, \quad e_3=\Diagethree, \text{ etc.}$$
By expanding the very top $R^*$-box in the diagram of $d_n$ and using \eqref{formula-semplificazioni} we obtain  
$$d_{n+1}= \bar b(b\delta +a) d_{n} + \bar a e_{n}.$$
Similarly, by expanding the top left $R$-box in the diagram of $e_n$ we obtain that
$$e_{n+1}=a(\bar a \delta + \bar b) e_n+ b d_{n+1}=(|a|^2\delta+a\overline{b}+\overline{a}b) e_n + |b|^2 (b\delta+a)d_n.$$
Therefore, it holds $Av_n=v_{n+1}$ where 
$$A:=  \left(\begin{matrix} 
\overline{b}(b\delta+a) & \overline{a}\\
|b|^2 (b\delta+a) & |a|^2\delta+a\overline{b}+\overline{a}b
\end{matrix}\right)
\text{ and } v_n=\begin{pmatrix} d_n\\ e_n\end{pmatrix}.$$
We obtain that $\varphi(\ga_{n+2})$ is the first component of the vector $A^n v_2$.

Denote by $$r_\pm=\frac{tr(A) \pm \sqrt \Delta}{2}$$ the roots of the characteristic polynomial $\chi_A(t)=t^2-tr(A) t+\det(A)$, and put $\Delta =(tr(A))^2-4\det(A)$ for its discriminant. 
Recall that we are considering only the real case and thus the curve $\CC_\R^\delta =\{(x,y)\in\R\times\R\; | \; \delta(x^2+y^2)+2xy = 1 \}.$
There exists an invertible matrix $S$ such that $A = S\left(\begin{matrix} 
r_+ & 0\\
0 & r_-
\end{matrix}\right)
 S^{-1}$
 or 
  $A = S\left(\begin{matrix} 
r_+ & 1\\
0 & r_+
\end{matrix}\right)
 S^{-1}$
and accordingly
$A^n= S\left(\begin{matrix} 
r_+^n &  0\\
0 & r_-^n
\end{matrix}\right) S^{-1}$
or 
$A^n= S\left(\begin{matrix} 
r_+^n & nr_+^{n-1}\\
0 & r_+^n
\end{matrix}\right) S^{-1}$
 for $n\geq 1$. 
This means that if $|r_\pm|<1$, then   $\lim_n A^nv_2=0$ and thus $\lim_n \varphi(\ga_n)=0.$	

\medskip
\noindent
{\bf Claim 1: If $\Delta<0$, then $\vert r_\pm\vert<1.$}\\
First of all we observe that $r_+=\overline{r_-}$, thus $\det(A)=r_+r_-=|r_+|^2=|r_-|^2\geq 0$. 
Without loss of generality we may suppose that $\det(A)\neq 0$ and that $a\in \R_+$, $a\geq |b|$. We see that
$$
\det(A) =ab(b\delta+a)(a\delta +b)=(b^2\delta+ab)(a^2\delta+ab)=x(1-x)
$$
where $x:=a(a\delta+b)>0$. 
Since $\det(A)>0$, we get that $0<1-x$, which in turn implies that $0<\det(A)<1$.
Therefore, $|r_\pm|<1$.

\medskip
\noindent
{\bf Claim 2: If $\Delta\geq 0$ and $tr(A)\geq 0$, then $\vert r_\pm\vert<1.$}\\
We claim that $\chi_A(1)>0$. Recall $tr(A)=1+ab$ and that
$$
\det(A)=ab(a\delta+b)(b\delta+a)=ab(ab\delta^2+1-ab)=ab(1+ab(\delta^2-1))\; .
$$
Then, $\chi_A(1)=1-(1+ab)+\det(A)=ab(1+ab(\delta^2-1)-1)=(ab)^2(\delta^2-1)>0$. 
Note that $ab\leq 1/2$ since $1=\delta(a^2+b^2)+2ab$ and thus $tr(A) = 1 + ab \leq 3/2 <2$ implying that $r_-=\frac{tr(A)-\sqrt \Delta}{2}\leq \frac{tr(A)}{2}<1.$
Since $tr(A)$ and $\Delta$ are non-negative we have that $r_-,r_+$ are real numbers such that $r_+=\vert r_+\vert \geqslant \vert r_-\vert.$
Moreover, $\chi_A(t)$ is negative if and only if $r_-\leq t\leq r_+$.
Since $r_- <1$ and $\chi_A(1)>0$ we necessarily have that $r_+<1$ and thus $\vert r_\pm\vert<1.$

\medskip
\noindent
{\bf Claim 3:
If $tr(A)<0$, then $\Delta<0$ implying that $|r_\pm|<1$ by Claim 1.}
Assume that $tr(A)<0$.
Since $tr(A) = 1 + ab$ we have that $ab<-1$ and thus $(ab)^2 > \vert ab \vert >1.$
Using the formula $\det(A) =  ab(1+ab(\delta^2-1))$ we obtain:
\begin{align*}
\Delta & = tr(A)^2 - 4\det(A) = 1 + (ab)^2 + 2ab - 4 ab(1+ab(\delta^2-1))\\
& \leq 1 - 2 ab - 3 (ab)^2 \text{ since } \delta^2\geq 2\\
& < 0 \text{ since } (ab)^2 > \vert ab \vert >1.
\end{align*}
This proves Claim 3.

Altogether, we obtain that $|r_\pm|<1$ for every cases implying that $\lim_n\varphi(\ga_n)=0.$
\end{proof}

We obtain the main result of this section.

\begin{theorem}\label{cor-AOV}
Let $(\pi,\scrH)$ be the unitary representation of $F$ constructed from the planar algebra $\TLJ(\delta)$ and a real couple $(a,b)\in \CC^\delta_\R$ with $a\neq 0\neq b.$
Then, for any $\varep>0$ and any finite dimensional subspace $\scrK\subset \scrH$ there exists $\gamma\in F$ such that 
$|\langle \pi_\gamma\eta, \zeta \rangle|<\varep \Vert \eta \Vert \Vert\zeta\Vert \text{ for any } \eta,\zeta\in \scrK.$
In particular, $\pi$ does not contain any finite dimensional subrepresentation.
\end{theorem}

\begin{proof}
We claim that for any $\varep>0$ and any $n\geq 1$ there exists $m\geq n$ and forests $p,q\in \cF(n,m)$ such that $\Phi(p)^*\circ \Phi(q) = c \id_{\CP_{2n-1}}$ with $|c|<\varep$.

Consider the sequence of elements $\gamma_k\equiv  x_0^k=\frac{s_k}{t_k}$ with trees $s_k, t_k$ as in Proposition \ref{prop:coefzero}.
Set 
$$
p_k= s_k \bullet |^{\bullet n-1} \text{ and } q_k= t_k \bullet |^{\bullet n-1}.
$$
We have that $\Phi(p_k), \Phi(q_k)\in \cF(\tilde n, 2\tilde n-1)$, where $\tilde n:=2n-1.$
Then
$$
\Phi(q_k)^*\circ \Phi(p_k)= \Phi(g_k)^* \circ\Phi(f_k) \bullet \vert^{\bullet \tilde n-1} = \langle \pi_{\gamma_k} \Omega, \Omega\rangle \; \vert^{\bullet \tilde n}=\varphi(\ga_k) \; \vert^{\bullet \tilde n}\; ,
$$
where we used the fact that 
$\dim\CP_1=1$ and hence, for all $x \in \CP_1$,
$$
\begin{tikzpicture}[baseline=-.1cm]
\draw (1.5,0.75)--(2,0.75)--(2,1.25)--(1.5,1.25)--(1.5,0.75);\node at (1.75,1) { $x$ };
\draw (1.75,1.25)--(1.75,1.75);
\draw (1.75,.75)--(1.75,.25);
\node at (2.75,1) { $=\tau_l(x)$ };
\draw (3.75,.25)--(3.75,1.75); 
\end{tikzpicture}
$$
$\tau_l$ being the normalized left trace of the planar algebra $\CP$.
This proves the claim since the sequence $(\varphi(\ga_k))_k$ tends to zero by Proposition \ref{prop:coefzero}.

We first prove the statement when $\scrK$ is equal to the $t$-subspace $\fH_t$ with $t$ a given tree.
Set $\gamma:= \frac{p t}{q t}\in F$ such that $\Phi(p)^*\circ \Phi(p) = c\id$ with $|c|<\varep$ and consider $\eta=\frac{t}{v}, \zeta=\frac{t}{w}\in \fH_t.$
We obtain
$$
|\langle\pi_{\gamma}\eta, \eta\rangle |=  |\langle \Phi(q)v , \Phi(p)w\rangle | = |c||\langle \eta , \zeta\rangle|<\varep \Vert \eta\Vert \Vert\zeta\Vert.$$ 

Now, fix $0<\varep <1$ and a finite dimensional subspace $\scrK\subset \scrH.$
Let $\eta_1,\cdots, \eta_k$ be an orthonormal basis of $\scrK.$
By density there exists a tree $t$ and a collection of unit vectors $\eta_1', \cdots , \eta_k'$ in $\fH_t$ such that $\Vert \eta_i - \eta_i'\Vert <\varep/k$ for any $1\leq i \leq k.$
By our previous argument, there exists $\gamma\in F$ such that $|\langle \pi_\gamma\eta, \zeta \rangle|<\varep \Vert \eta \Vert \Vert\zeta\Vert \text{ for any } \eta,\zeta\in \fH_t.$
Consider some unit vectors $\eta,\zeta\in \scrK$ and  expand them in the given orthonormal basis, namely $\eta = \sum_{i=1}^k a_i \eta_i$ and $\zeta = \sum_{j=1}^k b_j \eta_j$.
Put $\eta' = \sum_{i=1}^k a_i \eta_i'$ and $\zeta' = \sum_{j=1}^k b_j \eta_j'$  in $\fH_t$.
We observe  that $\Vert \eta - \eta'\Vert, \Vert \zeta - \zeta'\Vert <\varep$ and thus $\Vert \eta'\Vert, \Vert \zeta'\Vert < 1 + \varep.$
It follows that 
\begin{align*}
|\langle \pi_\gamma\eta, \zeta \rangle| & \leq |\langle \pi_\gamma(\eta-\eta'), \zeta \rangle| + |\langle \pi_\gamma\eta', (\zeta - \zeta')  \rangle| + |\langle \pi_\gamma\eta', \zeta' \rangle|\\
& \leq \Vert \eta - \eta'\Vert + \Vert \eta'\Vert  \Vert \zeta - \zeta'\Vert + \varep \Vert \eta'\Vert \Vert \zeta'\Vert \\
& \leq \varep + \varep(1+\varep) + \varep(1+\varep)^2 <7\varep.
\end{align*}
This concludes the proof of the first statement of the theorem.

Assume that $\scrK$ is the carrier space of a finite dimensional subrepresentation of $\pi.$
By the first assertion there exists $\gamma\in F$ such that $|\langle \pi_\gamma\eta, \zeta \rangle|<1/2 \Vert \eta \Vert \Vert\zeta\Vert$ for any $\eta,\zeta\in \scrK.$
If $\eta$ is a unit vector of $\scrK$, we have that $\sup\{ |\langle \pi_\gamma\eta, \zeta \rangle| : \zeta \in \scrK \text{ unit vector } \}=1$ which contradicts the previous inequality.
\end{proof}

In particular, we see that $\pi$ does not contain any one-dimensional representation and thus it does not contain the trivial representation $\pi_{1/\sqrt\delta,0}^\delta$, see Proposition \ref{prop:pitrivial}.

\begin{remark}\label{rem:infiniterep}
Let $\delta$ be as above. If $(a_t,b_t)$ is a path in $\CC^\delta_\R$ parametrized $0\leq t\leq 1$ and approaching $(1/\sqrt \delta,0)$ when $t$ tends to zero, then the net $\pi_t:=\pi_{a_t , b_t}^\delta$ tends to the trivial representation in  the Fell topology.
Therefore, $1_F$ is weakly contained in the direct sum $\oplus_{t_j} \pi_{t_j}$ for any countable subsequence $(t_j)_j$ tending to zero.
If the space of equivalence classes of $\pi_t$ is finite, then we can find a sequence $t_j$ tending to zero such that each $\pi_{t_j}$ is unitary equivalent to a single representation $\pi=\pi_{a,b}^\delta$ with $a\neq 0 \neq b.$
Then $1_F$ is weakly contained in the infinite direct sum of $\pi$ and thus $1_F$ is weakly contained in $\pi.$
Therefore, if $1_F$ is not weakly contained in $\pi_{a,b}^\delta$, then there are infinitely many pairwise non-unitarily equivalent representations in the class of Jones representations $\{\pi^\delta_{a,b} : (a,b)\in\CC^\delta_\R\}$.
\end{remark}

\subsection{Matrix coefficients which do not tend to zero}\label{sec:notzero}

We show that for any vector state and choice of real parameters $a,b$ there exists a sequence of group elements tending to infinity such that the corresponding coefficients do not tend to zero.
We start by showing that the vacuum vector state does not vanish everywhere outside the identity when $a,b$ are reals, which entails that none of the
associated representations is contained in a multiple of the left regular representation of $F$.

\begin{proposition}\label{prop:phix_0}
Consider a real pair $(a,b)\in\CC^\delta_\R$ and its associated representation $\pi:=\pi^\delta_{a,b}$ and vector state $\varphi:=\varphi^\delta_{a,b}.$
Then, 
we have that $\pi\not\subset \lambda^{\oplus \infty}.$
\end{proposition}

\begin{proof}
By Proposition \ref{prop:non-regular} it is enough to exhibit an element $g\in F$ such that $\varphi(g)\neq 0$.
In all but one case the element is the generator $x_0$.
Indeed,
assume that $\varphi(x_0)=0.$
Proposition \ref{phix0x1} tells us that $\varphi(x_0) = 1 + |ab|^2(1 - \delta^2)$
implying that $a\neq 0\neq b$.
We obtain that $b=\frac{\pm 1}{a\sqrt{\delta^2 - 1}}$. If $b=\frac{1}{a\sqrt{\delta^2 - 1}}$, then the normalization condition gives us
$$\Big(a^2 + \frac{1}{a^2(\delta^2-1)} \Big) \delta + \frac{2}{\sqrt{\delta^2 - 1}} = 1.$$
If $X:=a^2$, then we obtain the equation:
$$P(X):=\delta X^2 + \frac{ 2 - \sqrt{\delta^2 - 1}}{ \sqrt{\delta^2 - 1} } X + \frac{\delta}{\delta^2-1}=0.$$
But the discriminant of $P$ is equal to $-\frac{3(\delta^2 -1) + 4 \sqrt{\delta^2 - 1} }{ \delta^2 - 1 }<0.$
Hence, $P$ does not have any real root, a contradiction.\\
On the other hand, if $b=\frac{- 1}{a\sqrt{\delta^2 - 1}}$, by means of similar computations and also using the constraints on the values of $\delta$,    one can show that $\varphi(x_0)=0$ precisely when $\delta=\sqrt 2$ and $a=2^{1/4}$, $b=-2^{- 1/4}$ or $a=2^{- 1/4}$, $b=-2^{1/4}$.  
In the former case, 
we have 
that one half of $\varphi(g)$ coincides with the chromatic polynomial of $\Gamma(g)$ evaluated at $2$ for all $g\in F$  (cf.~Prop \ref{prop-chr-stab}) and thus $\varphi(x_0x_1)=1$. 
In the latter case we get $\varphi(\sigma_F(x_0x_1))=1$, see Propositions \ref{exchab} and \ref{phase}.
 \end{proof}

Consider the constant $C:=\sup_{\kappa\in F\setminus e} |\varphi(\kappa)|$ and recall that it is strictly positive by Proposition \ref{prop:phix_0}.

\begin{proposition}\label{prop:sequ-c}
For any vectors $\eta,\zeta\in \scrH$ we have the inequality:
$$\limsup_{\gamma\in F} | \langle \pi_\gamma \eta , \zeta \rangle | \geq C \sup_{\theta\in F} | \langle \pi_\theta \eta , \zeta \rangle | .$$
\end{proposition}

\begin{proof}
Consider unit vectors $\eta,\zeta\in \scrH$ and $\varep>0$.
We fix $\beta\in F$ such that $ | \langle \pi_\beta \eta , \zeta \rangle |\geq \sup_{\theta\in F} | \langle \pi_\theta \eta , \zeta \rangle |-\varep.$
By density we can assume that there exists a tree $t$ with $n$ leaves such that $\pi_\beta\eta$ and $\zeta$ belong to $\fH_t.$
Fix $\alpha\in F\setminus e$ such that $|\varphi(\alpha)| \geq C-\varep$ and consider some trees $s,r$ such that $\alpha =\frac{r}{s}.$
Denote by $\alpha_i$ the $i$-th shift of $\alpha$ and consider some trees $s_i,r_i$ satisfying $\alpha_i =\frac{r_i}{s_i}.$
Put $r_{i,n}:=r_i\bu |^{\bu n-1}$ the forest with $r_i$ and $n-1$ trivial trees on its right and $\ga_i=\frac{r_{i,n}\circ t}{ s_{i,n}\circ t}\in F$ where $s_{i,n}$ is defined similarly to $r_{i,n}.$
Observe that
\begin{align*}
| \langle \pi_{ \gamma_i \beta} \eta , \zeta \rangle | & = | \langle \Phi(s_{i,n}) \pi_\beta \eta , \Phi( r_{i,n} ) \zeta \rangle |\\
& = | \langle ([\Phi(r_{i})^*\circ \Phi(s_i)]\bullet |^{\bullet n-1})  (\pi_\beta \eta) ,  \zeta \rangle |\\
& = |\varphi(\alpha_i) | |\langle \pi_\beta \eta , \zeta\rangle|\\
& = |\varphi(\alpha) | |\langle \pi_\beta \eta , \zeta\rangle|\\
& \geq (C-\varep ) (\sup_{\theta\in F} | \langle \pi_\theta \eta , \zeta \rangle | - \varep )\\
& \geq C \sup_{\theta\in F} | \langle \pi_\theta \eta , \zeta \rangle | - 2\varep. \\
\end{align*}
Since $\alpha\neq e$ we have that $\alpha_i$ tends to infinity which implies that $\gamma_i\beta$ tends to infinity.
We obtain that $\limsup_{\gamma\in F} | \langle \pi_\gamma \eta , \zeta \rangle | \geq C \sup_{\theta\in F} | \langle \pi_\theta \eta , \zeta \rangle | - 2\varep$ for any $\varep>0$ which finishes the proof.
\end{proof}

The left regular representation of an infinite group satisfies that $\lim_n\psi(g_n)=0$ for any vector state $\psi$ and any sequence $(g_n)_n$ which goes to infinity.
This fact together with Proposition \ref{prop:sequ-c} imply the following.

\begin{theorem}\label{theo:lambdanotinpi}
For any   of $a,b\in \CC^\delta_\R$  the left regular representation of $F$ is not contained inside $\pi_{a,b}^\delta$.
Moreover, $\pi_{a,b}^\delta$ does not admit any coefficient vanishing at infinity.
\end{theorem}

Note that the proof of Proposition \ref{prop:sequ-c} still works for any representation $\pi$ arising from an irreducible subfactor planar algebra.
Furthermore, if this representation satisfies $\sup_{\kappa\in F\setminus e} |\varphi(\kappa)| > 0$ then we also have the conclusion of Theorem \ref{theo:lambdanotinpi}.

\begin{remark}
In \cite{DuMe14} it is shown that the only finite factor representations of $F$ are (the multiples of) the regular representation and one dimensional representations given by characters on the quotient by the commutator subgroup of $F$.
Therefore, Theorems \ref{cor-AOV} and \ref{theo:lambdanotinpi} imply that the representations $\pi^\delta_{a,b}$, with $a,b$ non-zero real numbers and any $\delta$, do not contain any finite factor subrepresentations.
\end{remark}

In the next section we will consider one case where the representation $\pi^\delta_{a,b}$ is known to be irreducible and thus of type I$_\infty$.

\section{Two representations that are not unitarily equivalent}\label{JonesSubgroup}\label{sec-two-reps}

Consider the loop parameter $\delta=\sqrt 2$ and the unitary representation $(\pi,\scrH)$ associated to the normalized Jones-Wenzl idempotent $R=2^{1/4} R_1 -2^{-1/4} R_2$.
The subgroup of $F$ that fixes the vacuum vector $\Omega$ in this representation is the Jones subgroup $\vec F$ introduced in \cite{Jo14}.
Consider the involution $\sigma:t\to 1-t$ of $[0,1]$ and its associated group isomorphism $\sigma_F:F\to F, g\mapsto \sigma g \sigma$.
We write $\tilde \pi:= \pi\circ \sigma_F$ and note that $\tilde \pi$ is unitary equivalent to $\pi_{\tilde R}$ with $\tilde R:=-2^{-1/4} R_1 + 2^{1/4} R_2$ by Remark \ref{exchab}.
The aim of this section is to show that $\pi$ and $\tilde \pi$ are not unitary equivalent.
We expect that the family of unitary representations $\{\pi^\delta_{a,b}, \delta,a,b\}$ contains infinitely many classes of unitary representations but so far $\pi$ and $\tilde \pi$ are the 
first nontrivial representations that we are able to distinguish directly up to unitary equivalence.
This also implies that $\pi$ and $\tilde\pi$ cannot be unitary equivalent to any representation $\pi^\delta_{a,a}$ or $\pi^\delta_{a,-a}$ i.e.~when $a=b$ or $a=-b$

Jones showed that the representation $\pi$ is unitary equivalent to the quasi-regular representation $\la_{F/\vec F}:F\act \ell^2(F/\vec F)$ implying that $\tilde \pi$ and $\la_{F/\sigma_F(\vec F)}$ are unitary equivalent.

Consider two subgroups $A,B<G$.
They are called commensurable if $A\cap B$ has finite index both in $A$ and in $B$.
Moreover, they are called quasi-conjugate if there exists $g\in G$ such that $A$ and $gBg^{-1}$ are commensurable.
Finally, recall that the commensurator $\Comm_G(A)$ is a subgroup of $G$ equal to the set of $g\in G$ such that $A\cap gAg^{-1}$ has finite index both in $A$ and in $gAg^{-1}$.
The following theorem is due to Mackey, see \cite{Mackey-book76} or \cite{Burger-delaHarpe97}.

\begin{theorem}
If $\Comm_G(A)=A$ and $\Comm_G(B)=B$, then the quasi-regular representations $\la_{G/A}$ and $\la_{G/B}$ are unitary equivalent if and only if $A$ and $B$ are quasi-conjugate.
\end{theorem}

Golan and Sapir proved that the commensurator $\Comm_F(\vec F)$ is equal to $\vec F$ \cite[Corollary 3]{GS17}
(and thus $\vec{F}$ is not almost normal in $F$). 
This implies that $\Comm_F(\sigma_F(\vec F))=\sigma_F(\vec F)$ as well.
Therefore, the representations $\pi$ and $\tilde \pi$ are unitary equivalent if and only if $\vec F$ and $\sigma_F(\vec F)$ are quasi-conjugate in $F$.

We recall a characterization of $\vec F$ due to Golan and Sapir \cite{GS17}.
For any dyadic rational $t\in (0,1)$ there exists a unique $n\geq 0$ and $a_1,\cdots,a_n\in \{0,1\}$ such that $t= \sum_{i=1}^n \frac{a_i}{2^i} + \frac{1}{2^{n+1}}$.
We write $t= .a_1\cdots a_n 1$.
Define $S$ as the set of dyadic rationals $t=.a_1\cdots a_n 1$ such that the set of $1\leq i \leq n$ satisfying $a_i=1$ is even.

\begin{theorem}\cite[Theorem 2]{GS17}
The subgroup $\vec F$ is the stabilizer of $S$ for the usual action $F\act [0,1]$,
i.e.~$\vec F = \{ g\in F : gS = S\}.$
\end{theorem}

Let us characterize in a similar way our subgroup $\sigma_F(\vec F)<F$.

\begin{lemma}
The subgroup $\sigma_F(\vec F)<F$ is the stabilizer of $\sigma(S)$.
Moreover, $\sigma(S)$ is the set of $t=.a_1\cdots a_n 1$ with $n\geq 0, a_1,\cdots,a_n\in \{0,1\}$ such that the set of $1\leq i\leq n$ satisfying $a_i=0$ is even.
\end{lemma}

\begin{proof}
Consider a dyadic rational $t\in (0,1)$.
It can be written as $t= \sum_{i=1}^n \frac{a_i}{2^i} + \frac{1}{2^{n+1}}$ with $a_1,\cdots,a_n\in \{0,1\}$ and $n\geq 0$.
Observe that 
\begin{align*}
\sigma(t) & = 1-t = \sum_{j=1}^\infty \frac{1}{2^j} - \left( \sum_{i=1}^n \frac{a_i}{2^i} + \frac{1}{2^{n+1}} \right)\\
& = \sum_{i=1}^n \frac{1- a_i}{2^i} + \frac{0}{2^{n+1}} + \sum_{j=n+2}^\infty \frac{1}{2^j}\\
& = \sum_{i=1}^n \frac{1-a_i}{2^i} + \frac{1}{2^{n+1}}.
\end{align*}
Therefore, $\sigma(.a_1\cdots a_n 1) = . (1- a_1)\cdots ( 1 - a_n ) 1$ which implies the desired characterization of $\sigma_F(S)$.
\end{proof}

We need a technical lemma to show that $\vec F$ and $\sigma_F(\vec F)$ are not quasi-conjugate.

\begin{lemma}\label{lem:d}
There exists $d\in \vec F \cap \sigma_F(\vec F)$ satisfying $\lim_{n\to \infty} d^n(t)=0$ for every $t\in(0,1)$.
\end{lemma}

\begin{proof}
Consider the pair of trees 
$$s:= (\vert \bu Y\bu \vert\bu \vert )\circ (Y\bu \vert\bu \vert )\circ (Y\bu \vert)\circ Y \text{ and } t:= (\vert\bu \vert \bu Y \bu \vert)\circ (\vert\bu \vert \bu Y) \circ (\vert \bu Y)\circ Y$$
This defines an element of the Thompson's group $g=\frac{t}{s}\in F$.
It sends the standard dyadic partition $\{ [0,1/2) , [1/2,3/4) , [3/4, 13/16) , [13/16, 7/8), [7/8,1]\}$ onto \\$\{ [0,1/8) , [1/8,3/16) , [3/16, 1/4), [1/4,1/2) , [1/2,1]\}$.
It can be defined by its action on dyadic rationals $t\in (0,1)$ as follows:
$$d(t)= \begin{cases}
.000a \text{ if } t = .0a\\
.0010a \text{ if } t= .10a\\
.0011a \text{ if } t= .1100a\\
.01a \text{ if } t= .1101a\\
.1a \text{ if } t=.111a
\end{cases},$$
where $a$ is any sequence of $0$ and $1$ with finitely many $1$.
From this description it is easy to check that $d$ stabilizes both the sets $S$ and  $\sigma(S)$.
Moreover, we have that $\lim_{n\to\infty}d^n(t)=0$ for any $t\in (0,1).$
\end{proof}

We are now able to prove the main result of this section.
The proof follows a similar strategy developed by Golan and Sapir in \cite{GS17}.

\begin{theorem}
The subgroups $\vec F$ and $\sigma_F(\vec F)$ of $F$ are not quasi-conjugate.
In particular, the representations $\pi$ and $\tilde \pi$ are not unitary equivalent.
\end{theorem}

\begin{proof}
We start by proving that the index $I:=[\vec F : \vec F\cap \sigma_F(\vec F)]$ is infinite.
Consider the element $c:=(x_0x_1)^{-1}$ that is in $\vec F$ where $x_n,n\geq 0$ is the classical set of generators of $F$.
If $I$ is finite, then there exists $r\geq 1$ such that $c^r\in \sigma_F(\vec F)$.
Recall from \cite[Remark 4.1.3]{GS17} that 
$$c(t) = \begin{cases} .00a \text{ if } t=.0a\\ .011a \text{ if } t=.110a\end{cases}.$$
Set $N(t):=| \{ i : a_i=0 \} |$ if $t=.a_1\cdots a_n 1$ and note that 
$$N(c^r(t)) =  \begin{cases} N(t) + r \text{ if } t=.0a\\ N(t)+ r-1 \text{ if } t=.110a\end{cases}.$$
This implies that $c^r$ does not stabilizes $\sigma(S)$ for any $r\neq 0$ and thus $c^r\notin \sigma_F(\vec F)$ for any $r\neq 0$.
Therefore, $I=\infty$.

Fix $g\in F$ and consider the index $I_g:=[ \sigma_F(\vec F) : \sigma_F(\vec F) \cap g \vec F g^{-1}]$.
If $g\in \vec F$, then $I_g=I=\infty$ by the previous argument.
Assume that $g\notin \vec F$.
Consider $d\in \vec F\cap \sigma_F(\vec F)$ as in Lemma \ref{lem:d}. 
Assume that $I_g$ is finite.
Then there exists $r\geq 1$ such that 
$d^r\in g\vec Fg^{-1}.$
Since $g$ is not in $\vec F$ there exists $t\in (0,1)$ that is not in $S$ such that $g(t)$ is in $S$.
Since $d$ is in $\vec F$ we have that $d^m(g(t))\in S$ for any $m\in \Z$.
By \cite[Lemma 4.14]{GS17} we have that there exists $\varep>0$ such that $g^{-1}(S\cap (0,\varep))$ is a subset of $S$.
Since $d$ satisfies the hypothesis of Lemma \ref{lem:d}, there exists $n\geq 1$ such that $d^{n r}(g(t))\in S\cap (0,\varep)$ and thus $g^{-1}(d^{n r }(g(t)))\in S$.
Since $t\notin S$ we obtain that $g^{-1}d^{n r}g\notin \vec F$ implying that $g^{-1}d^{r}g\notin \vec F$, a contradiction.
Therefore, $I_g$ is infinite for any $g \in F$ implying that $\vec F$ and $\sigma_F(\vec F)$ are not quasi-conjugate.

Our discussion at the beginning of this section implies that the representations $\pi$ and $\tilde \pi$ are not unitary equivalent.
\end{proof}

We have proved that $\pi$ and $\pi\circ\sigma_F$ are not unitary equivalent.
The remark done after Proposition \ref{exchab} tells us that when $a=b$ then the Jones representation $\pi^\delta_{a,a}$ is unitary equivalent to $\pi^\delta_{a,a}\circ\sigma_F$.
Since $\pi$ does not share this property we obtain the following corollary.

\begin{corollary}
Consider any $\delta$ equal to the square root of a non-trivial Jones index and the representation $\pi^\delta_{a,a}$ where $a=(2+2\delta)^{-1/2}.$
Then $\pi^\delta_{a,a}, \pi$ and $\pi\circ\sigma_F$ are mutually unitarily inequivalent.
\end{corollary}

\begin{remark}
We have seen in Remark \ref{rem:infiniterep} that if each non-trivial representation $\pi_{a,b}^\delta$ does not weakly contain the trivial representation, then there are necessarily infinitely many classes of unitary representations in this family.
If a representation weakly contains the trivial one then it is amenable (which would always be the case if the group $F$ were amenable). 
Looking at a quasi-regular representations such as $\la_{F/\Vec F}$, then it weakly contains the trivial representation if and only if the homogenous space $F/\Vec F$ is amenable which, in this case, is equivalent to saying that the representation $\la_{F/\Vec F}$ is amenable.
Even in this specific situation we do not know if $1_F$ is weakly contained in $\la_{F/\Vec F}$, which is an interesting open problem.
\end{remark}

\section{The stabilizer subgroup of the vacuum vector}\label{sec:stab}
In this section we restrict our attention to a subfamily of Jones representations constructed with the planar algebra $\TLJ$.
Fix a loop parameter $\delta$ and put 
$$R=\Big(\frac{\delta}{\delta^2-1}\Big)^{1/2} R_1 - \Big(\frac{1}{\delta^3-\delta}\Big)^{1/2} R_2$$ that is a scalar 
multiple of the Jones-Wenzl idempotent.

Denote by $(\pi,\scrH,\Omega)$ the associated representation with its vacuum vector and put $t:=\delta^2.$
We write $F_\Omega=F_\Omega(t)$ the subgroup of elements of $F$ that stabilize the vacuum vector.

We recall a formula that describes the vacuum vector state in term of the chromatic polynomial. It makes use of a well known argument that appears in \cite[Section 2]{Ka88} in a somewhat different setting (see also \cite{FK09}).
The corresponding statement for loopsided planar algebras can be found in \cite[Section 5.2]{Jo14}  (cf.~\cite[p.~28]{Jo99}). 
For the convenience of the reader we include a short proof that fits well with our formalism.
Recall that the chromatic polynomial of a graph evaluated at a natural number $n$ is equal to the number of proper $n$-colorations of the graph and is the only one satisfying this condition.

\begin{notation}
If $G$ is a graph and $t$ is a complex number, then we denote by $G(t)$ its chromatic polynomial evaluated in $t$.
\end{notation}

\begin{proposition}\label{prop-chr-stab}
Consider a group element $g$ of the Thompson's group $F$ described by the pair of trees $T_\pm$ having $n$ leaves.
Then 
$$
\langle \pi_g\Omega,\Omega\rangle=\frac{ \Ga(t)}{t(t-1)^{n-1}},
$$
where $\Ga=\Ga(T_+,T_-)$ is the graph associated to $T_\pm$ described in \cite{Jo14} and $t = \delta^2$.
\end{proposition}

\begin{proof}
Fix $g,n,T_\pm,$ and $\Ga:=\Ga(T_+,T_-)$ as in the statement.
Denote by $A$ the graph obtained as usual by concatenating vertically $T_-$ on the bottom and $T_+$ on top.
Recall that $\Ga$ has $n$ vertices and $2(n-1)$ edges.
The construction of $\Ga$ gives us a 
bijection between the edges of $\Ga$ and the vertices of $A$.
We associate to $A$ the tangle $T_A$ having $2(n-1)$ interior discs with 4 boundary points corresponding to each vertices and having one string on top and on the bottom.
For any subset $S\subseteq E(\Ga)$ consider the element $B_S\in \TLJ_1$ where we plug in $-\frac{1}{\sqrt{t}} \; {}^\cup_\cap$  in the interior of each  disc of $T_A$ corresponding to an edge of $S$ and $| \; |$ for the others and do not remove any closed curve.
We write $Z(B_S)$ for the scalar value of $B_S$ which consists of $t^{\Loop(S)/2}$, where $\Loop(S)$ is the number of closed curves of $B_S$.
Let $\Ga_S$ be the subgraph of $\Ga$ with the same vertex set as $\Gamma$ and edge set equal to $S\subset E(\Ga)$.
Write $k(S)$ for the number of connected components of $\Ga_S$ and $F(S)$ for the number of regions of $\Ga_S$.
Note that by Euler Formula we have that 
\begin{equation}\label{equa:Euler}
n-\vert S\vert + F(S)=1 + k(S).
\end{equation}
The element $B_S$ may be shaded 
(by convention the left part of $B_S$ is not coloured and thus the right side is coloured).
Observe that the number of uncoloured (resp. coloured) regions of $B_S$ is equal to $k(S)$ (resp.   $F(S)$).
 Then it is not difficult to see that
\begin{equation}\label{equa:loop}
k(S)+ F(S) = \Loop(S) +2
\end{equation}
since $B_S$ is a diagram with one vertical string and some closed curves.

By definition, we have that
\begin{align*}
\langle \pi_g\Omega ,\Omega\rangle & =\sum_{S\subseteq E(\Ga)} \left(\frac{\sqrt t}{t-1}\right)^{n-1} Z(B_S)\\
& =\sum_{S\subseteq E(\Ga)} \left(\frac{\sqrt t}{t-1}\right)^{n-1} \sqrt t ^{\Loop(S) } \left(\frac{-1}{\sqrt t}\right)^{|S|}\\
& =\sum_{S\subseteq E(\Ga)} (-1)^{|S|} \frac{t^{(\Loop(S)+n-1-|S| )/2} }{(t-1)^{n-1}} \\
& = \sum_{S\subseteq E(\Ga)} (-1)^{|S|} \frac{t^{(k(S)+F(S)+n-3-|S| )/2} }{(t-1)^{n-1}}  \text{ by the observation \eqref{equa:loop}   above,}\\
& = \sum_{S\subseteq E(\Ga)} (-1)^{|S|} \frac{t^{(2k(S)-2)/2} }{(t-1)^{n-1}}  \text{ by Euler Formula \eqref{equa:Euler}.}
\end{align*}
This implies the result via the Birkhoff and Whitney formula (see e.g.~\cite[Chapter X]{Bo98}) which claims that  the chromatic polynomial of $\Ga$ 
is equal to 
$\sum_{S\subset E(G)} (-1)^{\vert S\vert} X^{k(S)}.$
\end{proof}

The next result should be compared to \cite[Theorem 4.2]{AiCo1}. With $\delta$, $T_\pm$, $n$ as above, if $x$, $y$ lie on the hyperbola $(x-1)(y-1)=\delta^2$, 
by a similar argument one can indeed get the Tutte function
$T_g(x,y) := \frac{T_{\Gamma(T_+,T_-)}(x,y)}{(x+y)^{n-1}}$ on $F$ (here, $T_{\Gamma(T_+,T_-)}(x,y)$ denotes the Tutte polynomial) as the vacuum coefficient function of $\pi=\pi^\delta_{a,b}$, 
namely
$$\langle \pi_g \Omega, \Omega \rangle = T_g(x,y) \ , \; g \in F \ , $$
when $x < 1$, $y<1$ (so that $x+y <0$)
for $a=\sqrt{\frac{\delta}{(x+y)(y-1)}}$,
$b=-\sqrt{\frac{(y-1)}{\delta(x+y)}}$,
and when $x > 1$, $y>1$ 
for $a=\sqrt{\frac{\delta}{(x+y)(y-1)}}$,
$b=\sqrt{\frac{(y-1)}{\delta(x+y)}}$.
 With $x,y$ as above, 
 if in addition $x=y$,  
 the associated representation of $F$ is invariant under composition with $\sigma_F$ (up to unitary equivalence). Indeed, 
 in this case one has
$a=-b$ 
or
  $a=b$, 
  respectively, 
  and from Proposition \ref{phase}. 
The previous formulae can be rewritten directly in terms of $a$ and $b$. 
For any loop parameter $\delta$ and 
$(a,b)\in\CC_R^\delta$ satisfying $ab\neq 0$ we can express the vacuum state in terms of the Tutte polynomial as follows:
$$
\langle \pi_{a,b}^\delta(g)\Omega, \Omega\rangle = T_g\big(x,y\big) , \quad g \in F ,
$$
where $x =   (ab)^{-1}-\delta b/a-1$ and $y = \delta b/a+1$.
This provides a friendly combinatorial description of the vacuum state. It also suggests that the techniques developed in this section could be adapted to any non-degenerate choice of the real parameters $a,b.$

\begin{remark}\label{rem:Kauffman}
We point out another particular choice of the parameters which yields a vacuum vector state with a geometrical interpretation \cite{Jo14}.
For $n=3,4,5,\ldots $ set $A:=e^{i\pi (1\pm 1/n)/2}$ and $\delta := -A^2-A^{-2}=2\cos(\pi/n)$.
In this setting, if one choses $x=A/(-A^2-A^{-2})^{1/2}$ and $y=\bar A /(-A^2-A^{-2})^{1/2}$, that is $R=A/(-A^2-A^{-2})^{1/2} \; | \, |+ \bar A /(-A^2-A^{-2})^{1/2} \; {}^\cup_\cap$, then it can be easily verified that $R$ is normalized and that the vacuum vector state is equal (up to normalization) to the Kauffman bracket of a certain link. To be more precise, we have
$$
\varphi_{x,y}^\delta(g)=\frac{\langle L(T_+,T_-)\rangle}{( -A^2-A^{-2})^{n-1}}
$$ 
where $g=g(T_+,T_-)\in F$ for some trees $T_+, T_-$ with $n$ leaves and $L(T_+,T_-)$ is the knot/link associated to $g$ produced with Jones' construction relevant for $A$ being a root of unity \cite[Section 5.3]{Jo14}.
\end{remark}

We recover the description of the stabilizer subgroup $F_\Omega(t)<F$ of the vacuum vector $\Omega$ that is the $g\in F$ such that $\Ga(t)=t(t-1)^{n-1}$ where $g=(T_+,T_-)$, $\Ga=\Ga(T_+,T_-)$, and $T_\pm$ has $n$ leaves.
Note that if $t=2$, the subgroup $F_\Omega$ is the Jones subgroup $\Vec F$ studied in Section \ref{JonesSubgroup} and is equal to the set of $g=(T_+,T_-)$ having that $\Ga=\Ga(T_+,T_-)$ is bipartite as it was proven by Jones. 
We are interested in studying this subgroup $F_\Omega(t)$ for the other values of $t.$
We will prove that in many cases this subgroup is trivial.

\begin{remark}
Note that since $\pi$ is a unitary representation and $\Omega$ is of norm one we have that
\begin{equation}\label{equa:chromatic}
\vert\Ga(t)\vert\leqslant t(t-1)^{n-1} \end{equation}
for any Thompson graph $\Ga$ (the graph associated to any group element $g=(T_+,T_-)\in F$) with n vertices and for any index $\{4\cos(\pi/k)^2:k\geqslant 4\}\cup [4,\infty).$
As far as we know  it is an open question whether this inequality remains true for $t\geq 2$ and for any connected planar graph  $\Ga$  with $n$ vertices.
\end{remark}

In particular, the subgroup $F_\Omega(t)$ is trivial if for any Thompson graph $\Ga$ with $n$ vertices that is not a tree we have that $|\Ga(t)|<t(t-1)^{n-1}.$
Here are some well known and easy observations.

\begin{proposition}\label{prop:chromatic}
\begin{enumerate}
\item If equation \eqref{equa:chromatic} is true for any connected planar graphs for a fixed real $t>2$, then the subgroup $F_{\Omega}(t)$ is trivial.
\item If $|\Ga(t)|>0$ for any connected planar graph for a fixed real $t$, then  the subgroup $F_{\Omega}(t)$ is trivial.
\item Equation \eqref{equa:chromatic} is true for any natural numbers $k=t\geq 2$ and any connected planar graphs. 
 \item There are at most countably many values of $t$ such that $F_\Omega(t)$ is non-trivial.
\end{enumerate}
\end{proposition}

\begin{proof}
Proof of (1).
Consider a Thompson graph $\Ga$ with vertices $1,\cdots,n$.
Note that $\Ga=\La+e+f+n$, where $\La$ is obtained from $\Ga$ by removing the last vertex $n$ that has degree one or two and we wrote $e,f$ the edges.
Note that $\Ga=(\Ga-e) - \Ga/e=(t-1)\La -\La + \La/(ij)=(t-2)\La(t)+\La/(ij)$ where $i=\target(e),j=\target(f)$.
The inequality implies that 
$\vert\Ga(t)\vert\leqslant (t-2)t(t-1)^{n-2} + t(t-1)^{n-3}<t(t-1)^{n-1}$ since $t>2$.
In particular, $\langle\pi(g)\Omega,\Omega\rangle\neq 1$ when the Thompson graph  $\Ga$ associated to $g$ is not a tree (up to parallel edges). But this means that $g\neq 1$ and thus the subgroup is trivial.

Proof of (2).
Consider $t$ such that it is known that $\Ga(t)>0$ for any planar graph.
Fix a connected planar graph $\Ga$ that is not a tree.
Consider a connected planar graph $\Ga$.
It contains a spanning tree $T$ and thus $\Ga=T+(e_1,\cdots,e_l)$ where $e_i$ are some edges of $\Ga$.
We prove the result by induction on $l$.
Inizialization, $l=1$: we have $\Ga=T+e = T - T/e$.
Therefore, $\Ga(t)<T(t)$.
Suppose the result true for $l$, assume $\Ga= T+(e_1,\cdots,e_l,e_{l+1})$, and put $\La=T+(e_1,\cdots,e_l)$.
Then $\Ga=\La+e_{l+1}=\La - \La/e_{l+1}$. The same argument gives that $\Ga(t)<\La(t)<T(t)$.

Proof of (3).
Consider a connected planar graph $\Ga$, a natural number $k\geqslant 2$, and assume that $\Ga$ is not a tree.
Consider a spanning subtree $T\subset \Ga$ with the set vertex set.
Observe that the set of $k$-coloring of $\Ga$ is a subset of the $k$-coloring of $T$ which implies the inequality.

Proof of (4).
Consider a non-trivial group element $g\in F$ and its associated graph $\Ga$. If $g$ can be expressed by a pair of trees with $n$ leaves, then $\langle \pi^t_g\Omega,\Omega\rangle=\frac{ \Ga(t)}{t(t-1)^{n-1}}$, where $\pi^t$ is the representation associated with $t=\delta^2$ that we considered all along this section.
Therefore, $g\in F_\Omega(t)$ if and only if $\Ga(t) = t(t-1)^{n-1} $. It is well known that only a tree with $n$ leaves has its chromatic polynomial equal to the polynomial $x(x-1)^{n-1}$.
Moreover, if $g$ is non-trivial, then its associated graph $\Ga$ is not a tree.
Therefore, the set $X_g:=\{t>2: g\in F_\Omega(t)\}$ is necessarily finite since it is equal to some roots of a non-zero polynomial.
We obtain that $X=\bigcup_{g\in F, g\neq 1} X_g$ is a countable set that is equal to the set of $t$ such that $F_\Omega(t)$ is non-trivial.

\end{proof}

Birkhoff and Lewis proved that any planar graph has no real roots in $[5,\infty)$ and thus is strictly positive on this half-line implying that $F_\Omega(t)$ is trivial for $t>5$ by Proposition \ref{prop:chromatic}.
We present an elementary proof giving a slightly smaller lower bound.

\begin{theorem}\label{theo:large-t}
Equation \eqref{equa:chromatic} is true for any connected planar graph and any real number $t>4.63$.
Moreover, the inequality is strict if $\Ga$ is not a tree implying that $F_\Omega(t)$ is trivial for any $t>4.63.$
\end{theorem}

We start with a useful lemma.

\begin{lemma}
Consider a graph $\La$ and $k\geq1$ vertices of $\La$ written $\{1,\cdots, k\}$.
Let $0$ be a new vertex and let $\{ (01),\cdots, (0k) \}$ be some edges from $0$ to $j, 1\leq j\leq k$.
Denote by $\La_k$ the graph obtained by adding $0$ and those edges to $\La.$
We identify the graph $\La$ with its chromatic polynomial evaluated in $t$.
By removing-contracting edges we obtain the following formula for any $k$:
\begin{equation}\label{equa:Lambda_k}
\La_k= (t-k)  \La + \sum_{i<j} \La/(ij) - \sum_{i<j<k} \La/(ijk) + \sum_{i<j<k<l}\La/(ijkl) -\cdots,
\end{equation}
where $\La/(ij)$ is the graph $\La$ quotiented by $(ij)$ that is an edge between $i$ and $j$.
We write $\La/(ijk)$ the quotient by $(ij)$ and $(jk)$ and so on.
In particular,
\begin{align*}
\La_2 & = (t-2)\La + \La/(12)\\
\La_3 & = (t-3)\La + \big(\La/(12) + \La/(13) + \La/(23)\big) - \La/(123)\\
\La_4 & = (t-4)\La + [ \La/(12)+\La/(13)+\La/(14)+\La/(23)+\La/(24)+\La/(34) ] \\
& - [ \La/(123)+\La/(124)+\La/(134)+\La/(234) ] + \La/(1234).\\
\end{align*}
\end{lemma}

\begin{proof}
This can be easily proved by induction on $k$ and by using the graph $L_k=\La+((1k),(2k),\cdots,(k+1,k))$.
\end{proof}

\begin{proof}[Proof of the Theorem]
Fix $t>4.63$.
We prove the theorem by induction on the number of vertices $n$.
It is easy to prove it for $n=1$, $2$, $3$, $4$.
Assume it is true for any $k\leqslant n$ and consider a connected planar graph $\Ga$ with n vertices.
Since $\Ga$ is planar it admits a vertex of degree smaller or equal to 5 by Euler formula.
Therefore, there exists a planar graph $\La$ with $n-1$ vertices such that $\Ga=\La_k$ with $1\leqslant k\leqslant 5$.
First assume that $\La$ is connected.
Then any quotient of $\La$ is also connected and thus satisfied the inequality of the theorem.
We have that $$\| \La_k\|\leqslant \| t-k \| t(t-1)^{n-2} + k t(t-1)^{n-2} + \begin{pmatrix} k\\2\end{pmatrix} t(t-1)^{n-3}+\cdots.$$
We can show that for our $t > 4.63$, each of those polynomials evaluated in $t$ is strictly smaller than $t(t-1)^{n-1}$ when $k=2,3,4,5$.
If $k=1$ then $\|\Ga\|=(t-1)\|\La\|\leqslant (t-1)t(t-1)^{n-2}$.
Moreover, the inequality is strict if $\La$ is not a tree which is equivalent to have that $\Ga$ is not a tree.

Suppose now that $\La$ is not connected.
Since $\Ga$ is connected we have that $k\geqslant 2$ and $\Ga$ is a union of connected graphs $\Ga_1,\cdots,\Ga_j,j\leqslant k$ with one common vertex.
We can prove by induction on $j$ that $ \Ga = ( \frac{ t-1 }{ t } )^{ j-1 } \prod_{i=1}^j \Ga_i  $.
Moreover, $\sum_{i=1}^j n_i = n+j-1$ where $n_i$ is the number of vertices of $\Ga_i$.
Applying the inequality to each $\Ga_i$ we obtain that $\vert\Ga\vert\leqslant t(t-1)^{n-1}$ with equality if each of the $\Ga_i$ is a tree.
Each $\Ga_i$ is a tree implies that $\Ga$ is a tree.
This finishes the proof of the theorem.
\end{proof}

\begin{remark}
We know that the Jones subgroup $\vec F=F_\Omega(2)$ is non-trivial and that $F_\Omega(t)$ is trivial for $t=3,4$ and $t>4.63$. 
Thomassen and Perrett have recently proved that $\Ga(\tau+2)>0$ for all planar graphs $\Ga$ (see e.g.~\cite[Theorem 6.16, p.~100]{P16}) which implies that $F_\Omega(\tau +2)$ is trivial where $\tau=\frac{1 + \sqrt{5}}{2} = 2 \cos(\pi/5)$ is the golden ratio.
It seems very likely that all subgroups $F_\Omega(t)$ are trivial except the first one with $t=2$. However, except for $t=\tau+2$ and $t=3$ we have been unable to rule out the discrete series of $t=4\cos(\pi/k)^2, k\geq 5.$
The next Jones index after $2$ is $\tau^2=\tau+1$ and then $3$. 
Unfortunately, we cannot apply the argument for $\tau+2$ to the value $\tau+1$ as there are elements $g \in F$ with an edge $e$ of $\Gamma(g)$ such that $\Gamma(g)/e$ is the complete 4-graph $K_4$, and hence $(\Gamma(g)/e)(\tau^2) =\tau^2(\tau^2-1)(\tau^2-2)(\tau^2-3)< 0$. 
An example of this sort is given by
$$
\begin{tikzpicture}[baseline = -.1cm]
\draw (1,0) to [out=90,in=90] (2,0);
\draw (1,0) to [out=90,in=90] (3,0);
\draw (1,0) to [out=90,in=90] (5,0);
\draw (3,0) to [out=90,in=90] (4,0);
\draw (1,0) to [out=-90,in=-90] (2,0);
\draw (2,0) to [out=-90,in=-90] (3,0);
\draw (2,0) to [out=-90,in=-90] (4,0);
\draw (4,0) to [out=-90,in=-90] (5,0);
\node at (4.5,-0.5) {{\scriptsize $e$}};
\node at (-0.05,0) {$\Gamma(g)=$};
\end{tikzpicture}
$$
\end{remark}

\bibliographystyle{plain}

\end{document}